\documentclass{amsart}

\usepackage{graphicx}
 \usepackage{subcaption}
\usepackage[utf8]{inputenc}
\usepackage{nicefrac}
\usepackage{mathabx}
\usepackage{amsrefs}
\usepackage{esint}

\newtheorem{theorem}{Theorem}[section]
\newtheorem{lemma}[theorem]{Lemma}

\theoremstyle{definition}
\newtheorem{definition}[theorem]{Definition}
\newtheorem{example}[theorem]{Example}

\newtheorem{prop}[theorem]{Proposition}
\newtheorem{cor}[theorem]{Corollary}

\allowdisplaybreaks
\theoremstyle{remark}
\newtheorem{remark}[theorem]{Remark}

\numberwithin{equation}{section}

\newcommand{\A}{\mathcal{A}}

\newcommand{\E}{\mathcal{E}}

\newcommand{\dist}{\mathrm{dist}}
\newcommand{\diam}{\mathrm{diam}}
\newcommand{\supp}{\mathrm{supp}}
\newcommand{\dx}{\;\mathrm{d}x}
\newcommand{\dy}{\;\mathrm{d}y}
\newcommand{\dr}{\;\mathrm{d}r}

\newcommand*{\avint}{\mathop{\ooalign{$\int$\cr$-$}}}

\begin{document}

\title[The Biharmonic Alt-Caffarelli Problem]{The Biharmonic Alt-Caffarelli Problem in 2D}

\author{Marius Müller}
\address{Institut für Analysis, Universität Ulm, 89069 Ulm}
\email{marius.mueller@uni-ulm.de}
\thanks{The author is supported by the LGFG Grant (Grant no. 1705 LGFG-E) and would like to thank Anna Dall'Acqua and Fabian Rupp for helpful discussions.}
%





\begin{abstract}
We examine a variational free boundary problem of Alt-Caffarelli type for the biharmonic operator with Navier boundary conditions in two dimensions. We show interior $C^2$-regularity of minimizers and that the free boundary consists of finitely many $C^2$-hypersurfaces. With the aid of these results, we can prove that minimizers are in general not unique. We investigate radial symmetry of minimizers and compute radial solutions explicitly.
\end{abstract}
\maketitle

\section{Introduction}
\subsection{History and Context}
This article deals with a higher order version of the Alt-Caffarelli problem, which is a free boundary problem posed in \cite{Alt}. The classical first-order formulation can be understood as a variational Dirichlet problem with `adhesion' term. More exactly, the energy the authors consider is given by 
\begin{equation*}
\E_{AC}(u) := \int_\Omega |\nabla u |^2 \dx + |\{ x \in \Omega : u(x) > 0 \}|
\end{equation*}
 where $u \in W^{1,2}(\Omega)$ is such that $u - u_0 \in W_0^{1,2}(\Omega)$ for some given sufficently regular positive function $u_0$. Here, $|\cdot|$ denotes the Lebesgue measure and $\Omega \subset \mathbb{R}^n$ is some sufficiently regular domain. The two summands of $\E_{AC}$ impose competing conditions on minimizers: The Dirichlet term becomes small for functions that do not `vary too much' and the measure term (that we call \emph{adhesion term}) becomes small if the function is nonpositive in a large subregion of $\Omega$. Minimizers have to find a balance between these two terms.

The measure penalization can be understood as an adhesion to the zero level: Indeed, the lattice operations on $W^{1,2}$ imply that each minimizer of $\E_{AC}$ is nonnegative.  Given this, a minimizer $u$ divides $\Omega$ into two regions, namely $\{ u = 0 \}$, the so-called nodal set, and $\{ u > 0 \}$. The interface between the two regions is then a free boundary. Because of this structure, the Alt-Caffarelli problem is also called \emph{`adhesive free boundary problem'}.



More recently, the biharmonic Alt-Caffarelli problem, which is also our object of study, has raised a lot of interest, cf. \cite{Valdinoci} and  \cite{ValdinociSing}. Here the energy reads 
\begin{equation*}
\E_{BAC}(u) := \int_\Omega |\Delta u |^2 \dx + |\{ x \in \Omega : u(x) > 0 \}| \quad u \in W^{2,2}(\Omega) : u - u_0 \in W_0^{1,2}(\Omega),
\end{equation*}
defined for $u \in W^{2,2}(\Omega)$ that satisfies again $u- u_0 \in W_0^{1,2}(\Omega)$ for $u_0,\Omega$ as above.  From now on we shall also assume that $\Omega \subset \mathbb{R}^2$ since two-dimensionality is essential for our argument.  

The minimization with no derivatives prescribed at the boundary is a weak formulation of Navier boundary conditions, cf. \cite[Chapter 2]{Sweers}. If a minimizer is $u$ is sufficiently regular, one can obtain classical Navier boundary conditions, i.e. $`\Delta u = 0 $ on $\partial \Omega$'.

Just as in the first-order case, $\E_{BAC}$ consists of two competing summands: The first one measures roughly how much a function bends. The second one measures the positivity set.  Minimizers of $\E_{BAC}$ again have to find a balance between `not bending too much' and being nonpositive in a large subregion of $\Omega$. 

 As the authors of \cite{Valdinoci} point out, the structure of the problem is now fundamentally different. Due to the lack of a maximum principle, a minimizer $u$ divides $\Omega$ suddenly into three regions $\{ u= 0 \}$, $\{ u > 0 \}$ and $\{u < 0 \}$. And indeed, as \cite[Proposition B.1]{Valdinoci} highlights, the third region will actually be present. Having three regions means that one can get two interfaces, one between $\{ u > 0 \}$ and $\{ u = 0 \}$ and one between $\{ u= 0 \}$ and $\{ u < 0 \}$, at least in case that $\{ u = 0 \}$ is a 'fat' set with nonempty interior. 
 
A promising technique to examine the boundary is to look at the gradient of a minimizer $u$ on $\{ u= 0 \}$. Recall that in the classical Alt-Caffarelli problem, where one has  nonnegativity of $u$, one can infer that $\nabla u = 0$ at all interface points, at least provided that $u$ is appropriately smooth. The regularity of $u$ was discussed in \cite{Alt} and turned out to be sufficient for this conclusion. 

 
 The goal of this article is to show that $\{ u= 0 \}$ is a $C^2$-smooth manifold and $\nabla u \neq 0 $ on $\{u = 0\}$. Note that this behavior is exactly opposite to the first order problem, which is surprising. This also settles the aforementioned question of how the interfaces look like: There is only one interface of interest, namely the one between $\{u > 0 \}$ and $\{ u< 0 \}$, which is given by $\{u = 0 \}$. Moreover, the nodal set is nowhere 'fat', i.e. its Hausdorff dimension is at most one.
 
Our result can therefore be understood as an improvement of \cite[Theorem 1.10]{Valdinoci} and the following discussion in the special case of two dimensions.
Two-dimensionality is needed for our argument since it relies on the fact that every minimizer is semiconvex, cf. Lemma \ref{lem:semicon}, which we can prove with methods that do not immediately generalize to higher dimension.   

The fact that the gradient does not vanish on the free boundary makes the problem fundamentally different from the obstacle problem for the biharmonic operator, which has been studied in a celebrated article by Caffarelli and Friedman in 1979, see \cite{Friedman}. The article was trendsetting for the study of fourth order free boundary problems and gave way to striking recent results in this field, cf. \cite{Aleksanyan}, \cite{Okabe1}, \cite{Okabe2}.


Higher order adhesive free boundary problems have many applications in the context of mathematical physics, for example for the study of elastic bodies adhering to solid substrates, see \cite{Miura1} and \cite{Miura2}. Moreover, the square integral of the Laplacian can be thought of as a linearization of the well-known Willmore energy, see the introduction of \cite{Valdinoci} for more details.

\subsection{Model and Main Results}
For the entire article the given framework is the following.
\begin{definition}(Admissible Set and Energy)\label{def:adm} 
 Let  $\Omega \subset \mathbb{R}^2$ be an open and  bounded domain with $C^2$-boundary. Further, let $u_0 \in C^\infty(\overline{\Omega})$ be such that $(u_0)_{\mid \partial \Omega} \geq \delta > 0 $ for some $\delta > 0 $. Define
\begin{equation*}
\A(u_0) := \{ u \in W^{2,2}(\Omega) : u- u_0 \in W_0^{1,2}(\Omega) \} 
\end{equation*}
and $\E : \A(u_0) \rightarrow \mathbb{R}$ by 
\begin{equation*}
\E(u) := \int_\Omega (\Delta u)^2 \dx + |\{u > 0 \}| .
\end{equation*}
We say that $u \in \A(u_0)$ is a minimizer if 
\begin{equation*}
\E(u) = \inf_{w \in \A(u_0)} \E(w).
\end{equation*}
\end{definition}

\begin{remark}\label{lem:eximin}
Existence of a minimizer $u \in \mathcal{A}(u_0)$ is shown in \cite[Lemma 2.1]{Valdinoci} with standard techniques in the calculus of variations.
\end{remark}

\begin{remark} 
As $\Omega$ is sufficently regular to have a trace operator (see \cite[Theorem 6.3.3]{Willem})  and $u \in W^{2,2}(\Omega) \subset C^{0,\beta}(\overline{\Omega})$ for each $\beta  \in (0,1)$, we get that $u_{\mid \partial \Omega} = u_0$ pointwise. 
\end{remark}
As we mentioned, the main goal of the article is to show 

\begin{theorem}[Regularity and Nodal Set] \label{thm:1.1}
Let $u \in \A(u_0)$ be a minimizer. Then $u \in C^2(\Omega)\cap W^{3,2-\beta}_{loc}(\Omega)$ for each $\beta > 0 $  and there exists a finite number $N \in \mathbb{N}$ such that
\begin{equation}\label{eq:geb}
\{ u <  0 \} = \bigcup_{i = 1}^N  G_i,
\end{equation}
where $G_i$ are disjoint domains with $C^2$-smooth boundary.
Moreover, $\nabla u \neq 0 $ on $\partial \{ u< 0 \} = \{ u = 0 \}$ and $\{u = 0 \}$ has finite $1$-Hausdorff measure. Additionally, $u$ solves 
\begin{equation}\label{eq:bihame}
2\int_\Omega \Delta u \Delta \phi \dx = -  \int_{\{ u = 0 \}} \phi \frac{1}{|\nabla u|}\, \mathrm{d}\mathcal{H}^1 \quad \forall \phi \in W^{2,2}(\Omega) \cap W_0^{1,2}(\Omega).
\end{equation}
\end{theorem}
Let us remark that for smooth $\Omega$, one can remove "loc" in the $W^{3,2-\beta}$ regularity statement, see Section  \ref{sec:fwts} for details where also Navier boundary conditions are discussed. 
Let us formally motivate the term $\frac{1}{|\nabla u|} d\mathcal{H}^1$ in \eqref{eq:bihame}. It can be seen  as a `derivative' of the $|\{ u > 0 \}|$-term of the energy in the following way: By \cite[Prop.3, Sect.3.3.4]{EvGar} one has that  for each $f \in C^\infty(\mathbb{R}^n)$ with nonvanishing gradient, 
 \begin{equation}\label{eq:formmeas}
 \frac{d}{dt} | \{ f > t \} | = \int_{ \{ f = t \} } \frac{1}{|\nabla f |} \; \mathrm{d}\mathcal{H}^1  \quad \textrm{for  almost every $t \in \mathbb{R}$}.
 \end{equation}
 Theorem \ref{thm:1.1} will finally be proved in Section 6. 
 Section 3, 4 and 5 prepare the proof of the main theorem by showing some helpful properties of minimizers. Among those are semiconvexity, superharmonicity of the Laplacian and the blow-up behavior close to the nodal set. In Section 7 we show some estimates for the negativity region which underlines the importance of \eqref{eq:bihame} for applications and future research.
We will also show that minimizers are in general not unique, proved in Section 8.

\begin{theorem}[Non-Uniqueness of Minimizers]\label{thm:nonun}
There  exist $\Omega$ and $u_0$ as in Definition \ref{def:adm} such that $\E$ has more than one minimizer in $\A(u_0)$. 
\end{theorem}

The construction in the proof of this theorem depicts exactly one domain and one admissible boundary value for which minimizers are not unique. We do not think that it is impossible to obtain positive uniqueness results within certain ranges of initial values. Analysis of such is however beyond the scope of this article. 

The non-uniqueness relies on the following phenomenon: We choose $\Omega = B_1(0)$ and $u_0\equiv \iota$ to be a constant function. If the constant is small, we observe minimizers that are negative already really close to the boundary. We expect it to look roughly like a funnel, which grows steeply close to the boundary and has a round-off tip in the negative region. If however the constant is large, the minimizer is a constant function (which is then always positive). Therefore there has to be a limit case in which one can find minimizers with both shapes. 

To do so, we compute radial minimizers explicitly. The fact that there exists radial minimizers follows from Talenti's symmetrization principle, see \cite{Talenti} and Section \ref{sec:talenti} for details. The explicit computation also relies on the Navier boundary conditions, which will be discussed in Section 9.

%

\section{Preliminaries}
\subsection{Notation}
In the following we will fix some notation which we will use throughout the article. For a set $A \subset \mathbb{R}^n$ we denote its complement by $A^c := \mathbb{R}^n \setminus A$ and the interior of the complement by $A^C := \mathrm{int}(\Omega \setminus A)$. For a Lebesgue measurable set $E \subset \mathbb{R}^n$ we define the \emph{upper density} of $E$ at $x \in \mathbb{R}^n$ to be \begin{equation*}
\overline{\theta}(E,x) := \limsup_{r \rightarrow 0 + } \frac{|E \cap B_r(x)|}{|B_r(x)|} .
\end{equation*}
 We say that a point $x$ lies in the \emph{measure theoretic boundary} of $E$ if both $\overline{\theta}(E,x)$ and $\overline{\theta}(E^c,x)$ are strictly positive. The measure theoretic boundary of $E$ is denoted by $\partial^*E$. If $\alpha$ is a measure on a measurable space $(X, \mathcal{F})$ and $A \in \mathcal{F}$ then we define the \emph{restriction measure} $\alpha\llcorner_A : \mathcal{F} \rightarrow \mathbb{R}_+ \cup \{ \infty\} $ via $\alpha\llcorner_A(B) := \alpha(A \cap B)$. If $(X, \mathcal{F}) = (\mathbb{R}^n, \mathcal{B}(\mathbb{R}^n))$ is the Euclidean space endowed with the Borel-$\sigma$-Algebra and $U \subset \mathbb{R}^n$ is a Borel set, then we denote by $M (U)$ the set of Radon measures on $U$, see \cite[Section 1.1]{EvGar}. Moreover $\mathcal{H}^s$ denotes the $s-$dimensional Hausdorff measure on $\mathbb{R}^2$. 
 
\begin{definition}[The Hilbert Space $W^{2,2}(\Omega) \cap W_0^{1,2}(\Omega)$] \label{def:hilbers}
In this article, the Hilbert space $W^{2,2}(\Omega) \cap W_0^{1,2}(\Omega)$ is always endowed with the scalar product
\begin{equation*}
(u,v) := \int_\Omega \Delta u \Delta v \dx.
\end{equation*}
\end{definition}


\begin{definition}[Lebesgue Points] 
Let $1 \leq p < \infty$and  $f \in L^p_{loc}(\Omega)$. We say that $x_0 \in \Omega$ is a $p-$Lebesgue point of $f$ if 
\begin{equation} \label{eq:avliim}
f^*(x_0) := \lim_{r\rightarrow 0} \fint_{B_r(x_0)} f(y) \dy 
\end{equation} 
exists and 
\begin{equation*}
\lim_{r\rightarrow 0 } \fint_{B_r(x_0)} \left\vert f^*(x_0) -  f(x)\right\vert^p \dx  = 0 .
\end{equation*}
\end{definition}

%
%
%
%

\begin{definition}[Semiconvexity] 
Let $\Omega \subset \mathbb{R}^n$ be open, $f : \Omega \rightarrow \mathbb{R}$ be a function and $A \in \mathbb{R}$. We call $f$ $A$-semiconvex if for each $x_0 \in \mathbb{R}^n$ the map $x \mapsto f(x) + A|x-x_0|^2$ is convex.
\end{definition}

\begin{definition}(Superharmonic Functions)
Let $A \subset \mathbb{R}^n$ be open.  
A function $u : A \rightarrow \mathbb{R} \cup \{ - \infty, \infty\}$ is called \emph{superharmonic} if $u$ is lower semicontinuous  in $A$ and for each $x \in A$ and $r > 0 $ such that $\overline{B_r(x)} \subset A$ one has 
\begin{equation*}
u(x) \leq \frac{1}{\mathcal{H}^1(\partial B_r(x))} \int_{\partial B_r(x)} u(y) d\mathcal{H}^1(y) =: \fint_{\partial B_r(x)} u(y) dS_r(y). 
\end{equation*}
A function $u$ is called \emph{subharmonic} if $-u$ is superharmonic. 
\end{definition}

\subsection{Energy Bounds}

\begin{lemma}[Energy Bound for Minimizers] 
Let $u_0$ be as in Definition \ref{def:adm}. Then 
\begin{equation}\label{eq:infbound}
\inf_{w \in \A(u_0)} \E(w) \leq |\Omega|. 
\end{equation}
\end{lemma}
\begin{proof}
Let $w\in W^{1,2}(\Omega)$ be the unique weak solution of 
\begin{equation*}
\begin{cases}
\Delta w = 0 & \mathrm{in} \; \Omega, \\
w = u_0 & \mathrm{on} \; \partial \Omega.
\end{cases}
\end{equation*}
By elliptic regularity, $w-u_0  \in W^{2,2}(\Omega)\cap W_0^{1,2}(\Omega)$ and hence $w \in \A(u_0)$. By the maximum principle, $\inf_\Omega w \geq \inf_{\partial\Omega} u_0 \geq \delta >0 $. Hence $|\{w > 0 \} | = |\Omega|$. All in all 
\begin{equation*}
\E(w) = \int_\Omega (\Delta w)^2 \dx + |\{ w > 0 \}| = |\Omega|. \qedhere
\end{equation*}
\end{proof}

\begin{example}\label{ex:iotapos}
In general, the bound in \eqref{eq:infbound} is not sharp. We give an example of $\Omega$ and $u_0$ as in Definition \ref{def:adm} such that 
\begin{equation}\label{eq:inbel}
\inf_{w \in A(u_0)} \E(w) < |\Omega|.
\end{equation}
Suppose that $\Omega = B_1(0)$ and $u_0 \equiv C$ for some $C <\frac{1}{8\sqrt{2}}$. Further define $w(x) := 2C |x|^2 - C$ for $x \in B_1(0)$. One easily checks that $w \in \A(u_0)$. Now $\{w > 0 \} = B_1(0) \setminus \overline{ B_{\frac{1}{\sqrt{2}}}(0)}$ and $\Delta w \equiv 8C $. Hence 
\begin{equation*}
\E(w) = \int_{B_1(0)} 64C^2 dx + |B_1(0) \setminus \overline{ B_{\frac{1}{\sqrt{2}}}(0)}| = 64C^2 \pi  + \frac{1}{2} \pi = \left( 64C^2 + \frac{1}{2} \right) \pi ,
\end{equation*}
that is smaller that $\pi = |\Omega|$ by the choice of $C$. 
\end{example}
\begin{remark}\label{rem:iotafin}
We claim that for large constant boundary values, \eqref{eq:infbound} is sharp. Indeed, let $\Omega$ be as in Definiton \ref{def:adm} and fix a constant function $u_0 \equiv const$ such that $u_0 > C_\Omega \diam(\Omega)^\frac{1}{2} |\Omega|^\frac{1}{2}$, where $C_\Omega$ denotes the operator norm of the embedding operator $W^{2,2}(\Omega) \cap W_0^{1,2}(\Omega) \hookrightarrow C^{0, \nicefrac{1}{2}}(\overline{\Omega})$.
 If $u \in \A(u_0)$ is a minimizer then for each $x \in \Omega$ and $z \in \partial \Omega$ one has 
\begin{align*}
u(x)& \geq u_0 - |u(x)- u_0| = u_0 - |(u - u_0)(x)- (u-u_0)(z)|  \\ &  > C_\Omega \diam(\Omega)^\frac{1}{2} |\Omega|^\frac{1}{2} - ||u-u_0||_{C^{0,\nicefrac{1}{2}}} |x-z|^\frac{1}{2}
\\ & \geq C_\Omega \diam(\Omega)^\frac{1}{2} |\Omega|^\frac{1}{2}- C_\Omega ||u-u_0||_{W^{2,2} \cap W_0^{1,2}} \diam(\Omega)^\frac{1}{2} \geq 0, 
\end{align*}
since $||u-u_0||_{W^{2,2}\cap W_0^{1,2}} = || \Delta u- \Delta u_0||_{L^2} = ||\Delta u ||_{L^2} \leq \sqrt{\E(u)} \leq \sqrt{|\Omega|}$. Therefore, all minimizers are positive, which means in particular that \eqref{eq:infbound} is sharp and the unique minimizer is given by the weak solution of 
\begin{equation*}
\begin{cases}
\Delta u = 0 & \mathrm{in} \; \Omega
\\
u = u_0 & \mathrm{on} \; \partial \Omega,
\end{cases}
\end{equation*}
which is $u \equiv u_0$. 
\end{remark}

\begin{remark}\label{rem:nontriv}
If $\inf_{w \in A(u_0)} \E(w) < |\Omega|$ and $u \in \A(u_0)$ is a minimizer then $\{ u= 0 \}$ cannot be empty. Indeed, if it were empty then $\{ u > 0 \} = \Omega$ by the embedding $W^{2,2}(\Omega) \subset C(\overline{\Omega})$. A contradiction. 
\end{remark}

\subsection{Variational Inequality}
In the rest of this section we derive that each minimizer $u$ is biharmonic on $\{ u > 0\} \cup \{ u < 0 \}$ and $\Delta u $ is weakly superharmonic on the whole of $\Omega$. The techniques used are standard perturbation arguments. We also draw some first conclusions about regularity of $u$.  

\begin{lemma}[Biharmonicity away from Free Boundary] \label{lem:biham}
Let $u \in \A(u_0)$ be a minimizer. Further, let $\phi \in C_0^\infty(\{u > 0 \})$ or $\phi \in C_0^\infty(\{ u < 0 \})$ or $\phi \in W_0^{1,2}(\Omega) \cap W^{2,2}(\Omega)$ with compact support in $\{ u >0 \}$. Then
\begin{equation}\label{eq:biiham}
\int_\Omega \Delta u \Delta \phi \dx  = 0 .
\end{equation}
In particular, $u \in C^\infty(\{ u >0 \}) \cup C^\infty( \{ u < 0 \})$ and $\Delta^2 u = 0 $ in $\{ u > 0 \} \cup \{ u <0 \}$ .
\end{lemma}
\begin{proof}
We show \eqref{eq:biiham} only for $\phi \in W_0^{1,2}(\Omega) \cap W^{2,2}(\Omega)$ with compact support in $\{ u >0 \}$. The other cases are similar. Since $u \in W^{2,2}(\Omega) \subset C(\overline{\Omega})$ and $\mathrm{supp}(\phi)$ is compact in $\{u > 0 \}$ there exists $\theta > 0 $ such that $u \geq \theta$ on  $\mathrm{supp}(\phi)$. In particular, for $t $ sufficiently small we get $\{u > 0 \} = \{ u + t\phi > 0 \} $. For such fixed $t$ one has 
\begin{equation*}
\frac{\E(u) - \E(u+t\phi) }{t} =2 \int_\Omega \Delta u \Delta \phi \dx + t \int_\Omega
 (\Delta \phi)^2 \dx .
\end{equation*}  
From the right hand side we infer that $t \mapsto \E(u+t\phi)$ is differentiable at $t = 0$. Using this and the fact that $u$ is a minimizer we obtain 
\begin{equation*}
0 = \frac{d}{dt}_{\mid_{t= 0}}  \E(u+t\phi) = 2 \int_\Omega \Delta u \Delta \phi \dx .
\end{equation*}
By Weyl's lemma $\Delta u $ is harmonic in $\{u> 0 \}$ and hence $C^\infty(\{u> 0 \})$. The claim follows.
\end{proof}

\begin{cor}[A Neighborhood of the Boundary] \label{cor:guterrand}
Let $u \in \A(u_0)$ be a minimizer and $\delta$ be as in Definition \ref{def:adm}. Then there exists $\epsilon_0 > 0 $ such that
$\Omega_{\epsilon_0} := \{ x \in \Omega : \mathrm{dist}(x, \partial \Omega) < \epsilon_0 \}$ has $C^2$-boundary, $u \in C^\infty(\Omega_{\epsilon_0})$ and $\Delta^2 u = 0$, as well as  $u \geq \frac{\delta}{2} $ in $ \Omega_{\epsilon_0} $. 
\end{cor}
\begin{proof}
Let $\delta$ be as in Definition \ref{def:adm}. 
Due to the uniform continuity of $u$, there exists $\epsilon^* >0 $ such that $u(x) > \frac{\delta}{2}$ whenever $\dist(x, \partial \Omega) < \epsilon^*$.  
 Because of \cite[Lemma 14.16]{Gilbarg} there is $\epsilon' > 0$ such that $\epsilon \leq \epsilon'$ implies that $\Omega_{\epsilon} := \{ x \in \Omega : \mathrm{dist}(x, \partial \Omega) < \epsilon \} $ has $C^2$-boundary.
The claim follows taking $\epsilon_0 := \min\{ \epsilon^* , \epsilon'\} $ and using Lemma \ref{lem:biham}.  
\end{proof}
\begin{remark}
Since it is needed very often, we will use the notation $\Omega_{\epsilon_0}$ from now on without giving further reference to Corollary \ref{cor:guterrand}. 
\end{remark}

\begin{lemma}[Euler-Lagrange-Type Properties] \label{lem:EL}
Let $u \in \A(u_0) $ be a minimizer. Then  for each $\phi \in W^{2,2}(\Omega) \cap W_0^{1,2}(\Omega)$ such that $\phi \geq 0 $ one has  
\begin{equation}\label{eq:subbiham}
\int_\Omega \Delta u \Delta \phi \dx \leq 0
\end{equation}
and
\begin{equation}\label{eq:meas}
\limsup_{\epsilon  \rightarrow 0+ } \frac{| \{ 0< u < \epsilon \phi \} | }{\epsilon}\leq  2\E(u) ^\frac{1}{2} ||\Delta \phi ||_{L^2}. 
\end{equation}
\end{lemma}
\begin{proof}
Set $\psi:= - \phi$. Then one has 
\begin{align}\label{eq:EL}
0 & \geq \limsup_{\epsilon \rightarrow 0 + } \frac{\E(u) - \E(u+\epsilon\psi) }{\epsilon} 
 \\ & =\limsup_{\epsilon \rightarrow 0 } - 2\int_\Omega \Delta u \Delta \psi \dx - \epsilon \int_\Omega (\Delta \psi)^2 \dx + \frac{ |\{0 < u < -\epsilon \psi \}|}{\epsilon}
\nonumber \\ & = - 2\int_\Omega \Delta u \Delta \psi \dx + \limsup_{\epsilon \rightarrow 0 + }  \frac{ |\{0 < u < -\epsilon \psi \}|}{\epsilon} \nonumber.
\end{align}
Since $\psi \leq 0$, we can first estimate the measure term from below by zero to obtain 
\begin{equation*}
0 \geq  -   \int_\Omega \Delta u \Delta \psi \dx = \int_\Omega \Delta u \Delta \phi \dx \quad \forall \phi \in W^{2,2}(\Omega) \cap W_0^{1,2}(\Omega) : \phi \geq 0,
\end{equation*}
that is \eqref{eq:subbiham}. Going back to \eqref{eq:EL}  and using the Cauchy-Schwarz inequality we find
\begin{equation*}
\limsup_{\epsilon \rightarrow 0 + } \frac{|\{0 < u < \epsilon(-\psi) \}| }{\epsilon } \leq 2 \int \Delta u \Delta \psi  \dx \leq 2 ||\Delta u||_{L^2} ||\Delta \psi||_{L^2} \leq 2\sqrt{\E(u)} || \Delta \psi ||_{L^2}   ,
\end{equation*} 
from which \eqref{eq:meas} follows again replacing $\phi := - \psi$. 
\end{proof}
\begin{cor}[Subharmonicity] \label{lem:corsubham} 
Let $u \in \A(u_0)$ be a minimizer. Then $(\Delta u)^* \geq  0 $ at every $1-$ Lebesgue point of $\Delta u$. In particular, $u$ is subharmonic. 
\end{cor}
\begin{proof}
Fix $x \in \Omega$ and let $r \in (0, \mathrm{dist}(x, \partial \Omega))$ be  arbitrary. Denote by $\phi_r$ the weak $W_0^{1,2}$-solution of 
\begin{equation*}
\begin{cases} 
\Delta \phi_r = \frac{1}{|B_r(x)| } \chi_{B_r(x)}  & \mathrm{in} \; \Omega \\ 
\phi_r = 0 & \mathrm{on} \;  \partial \Omega
\end{cases} .
\end{equation*}
By elliptic regularity, $\phi_r \in W^{2,2}(\Omega) \cap W_0^{1,2}(\Omega)$. By the maximum principle, $\phi_r \leq 0 $ a.e.. By \eqref{eq:subbiham}
\begin{equation*}
0 \leq \int_\Omega \Delta u \Delta \phi_r \dy = \fint_{B_r(x)} \Delta u \dy.
\end{equation*}
If $x$ is a Lebesgue point of $\Delta u $, we can let $r \rightarrow 0 $ and find that  $(\Delta u)^*(x) \geq  0$. Since $u \in W^{2,2}(\Omega)$, almost every point is a Lebesgue point and hence $\Delta u \geq 0 $ a.e.. Since $u$ is furthermore continuous, we get that $u$ is subharmonic, see \cite[Theorem 4.3]{Serrin}.
\end{proof}
\begin{cor}[Positive Part Near Free Boundary]\label{cor:posdir}
Let $u \in \A(u_0)$ be a minimizer. Then 
\begin{equation*}
\limsup_{\epsilon \rightarrow 0 } \frac{|\{0 < u < \epsilon\}|}{\epsilon} < \infty.
\end{equation*}
\end{cor} 
\begin{proof}
Choose $\phi^* \in C_0^\infty(\Omega)$ such that $0 \leq \phi^* \leq 1 $ and $\phi^* \equiv 1 $ on $\Omega_{\epsilon_0}^C$, which is defined as in Lemma \ref{lem:biham}. Then $u(x) \geq \frac{\delta}{2}$ for all $x \in \Omega_{\epsilon_0}$, where $\delta$ is given by Definition \ref{def:adm}. Therefore note that 
\begin{equation*}
\lim_{\epsilon \rightarrow 0 } \frac{|\{ 0 < u < \epsilon \} \cap \Omega_{\epsilon_0}| }{\epsilon} = 0. 
\end{equation*} 
By \eqref{eq:meas}
\begin{align*}
\limsup_{\epsilon \rightarrow 0 } \frac{|\{0 < u < \epsilon\}|}{\epsilon} & = \limsup_{\epsilon \rightarrow 0 } \frac{|\{0 < u < \epsilon\} \cap \Omega_{\epsilon_0}^C|}{\epsilon} \\ &  \leq \limsup_{\epsilon \rightarrow 0 } \frac{|\{0 < u < \epsilon\phi^*\}|}{\epsilon} \leq 2 \sqrt{\E(u) }||\Delta \phi^*||_{L^2} < \infty. \qedhere 
\end{align*}
\end{proof}
\begin{cor}[The Biharmonic Measure] Let $u \in \A(u_0)$ be a minimizer. 
Then there exists a finite Radon measure $\mu \in M(\Omega)$ such that $\mathrm{supp}(\mu) \subset \{ u = 0 \}$ and 
\begin{equation}\label{eq:bihammeas}
2\int_\Omega \Delta u \Delta \phi \dx  =-  \int_\Omega \phi \; \mathrm{d}\mu \quad \forall \phi \in W_0^{2,2}(\Omega) .
\end{equation} 
\end{cor}
\begin{proof}
Define $L: C_0^\infty(\Omega) \rightarrow \mathbb{R}$ by 
\begin{equation*}
L(\phi) := - 2\int_\Omega \Delta u \Delta \phi \dx . 
\end{equation*}
The map $L$ is linear and satisfies  $L(f) \geq 0 $ for each $f \geq 0 $ by \eqref{eq:subbiham}. By the Riesz-Markov-Kakutani Theorem (see \cite[Corollary 1, Section 1.8]{EvGar}) we infer that there exists a (not necessarily finite) Radon measure $\mu \in M(\Omega)$ such that 
\begin{equation*}
L(\phi) = \int_\Omega \phi \; \mathrm{d}\mu .
\end{equation*}
Furthermore, by Lemma \ref{lem:biham} we have that $L(\phi) = 0$ for each $\phi \in C_0^\infty(\{ u > 0 \}) \cup C_0^\infty( \{ u < 0 \} ) $. Since $\mu $ is Radon, this implies that $\mu( \{ u > 0 \} ) = \mu( \{ u < 0 \} ) = 0 $. Since $\{u > 0 \} $ and $\{ u < 0 \} $ are open by continuity of $u$, we have $\mathrm{\supp}(\mu) \subset \{ u = 0 \}$. However, since $u_{\mid_{\partial \Omega}} \geq \delta > 0 $ by Definition \ref{def:adm}, $\{u = 0 \} $ is compactly contained in $\Omega$. Hence $\mu(\Omega) = \mu(\{ u= 0 \} ) < \infty$ since $\mu$ is finite on compact subsets of $\Omega$.  It remains to show that \eqref{eq:bihammeas} holds for $\phi \in W_0^{2,2}(\Omega)$, but this holds because of density and the fact that $W_0^{2,2}(\Omega) \subset C(\overline{\Omega})$. 
\end{proof}
\begin{remark}
Note that for $\phi \in W_0^{2,2}(\Omega)$, \eqref{eq:bihammeas} holds only for the continuous representative of $\phi$. The precise representative is important since $\mu$ may not be absolutely continuous with respect to the Lebesgue measure. 
\end{remark}
From now on, whenever we address a minimizer $u \in \A(u_0)$, $\mu_u$ or in case of nonambiguity $\mu$ denotes the measure that satisfies \eqref{eq:bihammeas}. 

\begin{lemma}[Local BMO-regularity]\label{lem:bihmmeasmore}
Let $u \in \mathcal{A}(u_0)$ be a minimizer. Then $\Delta u \in BMO_{loc}(\Omega) \subset L^q_{loc}(\Omega), q \in [1, \infty)$ and \eqref{eq:bihammeas} holds true also for $\phi \in W_0^{2,p}(\Omega_{\epsilon_0}^C)$ for each $p \in (1,2)$. 
\end{lemma}
\begin{proof}
For the assertion that $\Delta u \in BMO_{loc}(\Omega)\subset L^q_{loc}(\Omega), q \in [1, \infty)$ we refer to \cite[Theorem 1.1]{Valdinoci}. Now fix $\phi \in W_0^{2,p}(\Omega_{\epsilon_0}^C)$. Since $\Omega_{\epsilon_0}^C$ has $C^2$-boundary by Corollary \ref{cor:guterrand} we obtain by Sobolev embedding that $\phi \in C(\overline{\Omega_{\epsilon_0}^C})$ and that there exists a sequence $(\phi_n)_{n = 1}^\infty \subset C_0^\infty(\Omega_{\epsilon_0}^C)$ that is convergent to $\phi$  in $W^{2,p}(\Omega_{\epsilon_0}^C)$ and in $C(\overline{\Omega_{\epsilon_0}^C})$. From this and the fact that \eqref{eq:bihammeas} holds for all $\phi_n$ one can infer that it also holds for $\phi$.
\end{proof}

\begin{remark}\label{rem:ceins}
In particular the previous Lemma implies that each minimizer lies in $C^1(\Omega)$.
\end{remark}
\section{Regularity and Semiconvexity}

In this section we will study regularity and some properties of the minimizer, in particular the set of non-$1-$Lebesgue points of $D^2u$. We will expose a singular behavior of the Laplacian at all those points. Moreover we prove that minimizers are semiconvex, which can also be seen as a regularity property, having Aleksandrov's theorem in mind. 

For our arguments, we need some remarkable facts about the fundamental solution in two dimensions that were already discovered and applied to the biharmonic obstacle problem by Caffarelli and Friedman in \cite[e.g. Equation (6.3)]{Friedman}. 
\begin{lemma}[{Fundamental Solution of the Biharmonic Operator, cf. \cite[Section 7.3]{Mitrea}}] \label{lem:green}
Define $F : \mathbb{R}^2 \times \mathbb{R}^2 \setminus \{(x,x): x \in \mathbb{R}^2 \} \rightarrow \mathbb{R}$ via 
\begin{equation*}
F(x,y) :=  \frac{1}{8\pi} |x-y|^2 \log |x-y| .
\end{equation*}
Then $F$ satisfies $\Delta^2 F(x,\cdot) = \delta_x$ on $\mathbb{R}^2$, where $\delta_x$ denotes the Dirac measure of $\{x\}$. Then for each $\beta  \in (0,1] $ one has that $F(\cdot,y) \in W^{3,2-\beta}_{loc} (\mathbb{R}^2)$  for each $y \in \mathbb{R}^2$. Moreover, for all $(x,y) \in \mathbb{R}^2 $ such that $x \neq y$ one has
\begin{align}\label{eq:ablgreen}
\nabla_x F(x,y) & = - \nabla_y F(x,y) = \frac{1}{8\pi}(2 \log|x-y| + 1) (x-y), \\
\partial_{x_ix_i}^2 F(x,y)&  = \frac{1}{8\pi} \left( 1+ 2 \frac{(x_i-y_i)^2}{|x-y|^2} +2 \log |x-y| \right) \quad i = 1,2 , \label{eq:seconder}  \\
\partial^2_{x_1x_2} F(x,y)& = \frac{1}{4\pi} \frac{(x_1-y_1) (x_2-y_2) }{|x-y|^2}. 
\end{align}
In particular, 
\begin{equation}\label{eq:laplaci}
\Delta_x F(x,y) = \frac{1}{2\pi} \left( \log |x-y| + 1 \right), 
\end{equation} 
and $\partial_{x_1x_1}^2F(x,\cdot) - \partial_{x_2x_2}^2F(x,\cdot), \partial_{x_1x_2}F(x, \cdot)\leq \frac{3}{8\pi}$  on $\mathbb{R}^2 \setminus\{x\}$ for each $x \in \mathbb{R}^2$. Moreover, there is $C > 0 $ such that 
\begin{equation}\label{eq:thidbound}
|D_x^3F(x,y)| \leq \frac{C}{|x-y|} \quad \forall y \in \mathbb{R}^2 \setminus \{ x \} .
\end{equation}
\end{lemma}

\begin{lemma}\label{lem:greensubham}
Let $x_0,y \in \mathbb{R}^2$ and
\begin{equation*}
H(r) := - \frac{1}{8\pi}\fint_{B_r(x_0)} \log|x-y| dx.
\end{equation*}
Then $H$ is decreasing on $(0,\infty)$ and its pointwise limit as $r \rightarrow 0$ is given by $- \frac{1}{8\pi}\log|x_0 - y|$ with the convention that $- \log 0 := \infty$
\end{lemma}
\begin{proof}
The claim follows directly from \cite[Proposition 4.4.11(6)]{Berenstein} and \cite[Proposition 4.4.15]{Berenstein}.
\end{proof}
The following result is very similar to crucial observations in \cite{Friedman}. 
\begin{lemma}[Biharmonic Measure Representation, Proof in Appendix \ref{sec:tecprf}] \label{lem:bihammeasrep}

Let $u \in \A(u_0)$ be a minimizer and $\mu$ be as in \eqref{eq:bihammeas}. Further let $\Omega_{\epsilon_0} $ be as in Corollary \ref{cor:guterrand}. Then there exists $h \in C^\infty(\overline{\Omega_{\epsilon_0}^C}) $ such that 
\begin{equation*}
u(x) = - \frac{1}{2}\int_\Omega F(x,y) \; \mathrm{d}\mu(y) + h(x)  \quad \forall  x \in \Omega_{\epsilon_0}^C,
\end{equation*}
where $F$ is the same as in Lemma \ref{lem:green}. 
\end{lemma}

The explicit representation of the minimizer will help to prove a first regularity result. The method used here is explained in the following lemma, whose proof is very straightforward by the definition of a weak derivative and Fubini's theorem.
\begin{lemma}[Kernel Operators with Measures]\label{lem:kernelmeas}
Let $\Omega\subset \mathbb{R}^n$ be open and bounded and $1 \leq p < \infty$. Let $\alpha$ be a finite Borel measure on $\Omega$ and let $\lambda$ denote the $n$-dimensional Lebesgue measure on $\Omega$. Let $H :\Omega \times \Omega \rightarrow \overline{\mathbb{R}}$ be a Borel measurable function on $\Omega \times \Omega$ such that 
\begin{enumerate}
\item $(x,y) \mapsto H(x,y) \in L^p(\lambda \times \alpha )$
\item For each $y \in \Omega$, $x \mapsto H(x,y)$ is weakly differentiable with $\Omega\times \Omega$-Borel measurable weak derivative $\nabla_x H(x,y)$.
\item $(x,y) \mapsto \nabla_x H(x,y) \in L^p(\lambda \times \alpha) $. 
\end{enumerate}
Then $A(x) := \int_\Omega H(x,y) \; \mathrm{d}\alpha(y) $ lies in $W^{1,p}(\Omega)$ and its weak derivative satisfies 
\begin{equation}\label{eq:abliden}
\nabla A(x) = \int_\Omega \nabla_x H(x,y) \; \mathrm{d}\alpha(y) .
\end{equation}
\end{lemma}
Using induction and the previous lemma, one easily obtains the following higher order version. 
\begin{cor}[Higher Order Derivatives]\label{cor:regreg}
Let $\Omega \subset \mathbb{R}^n$ be open and $1 \leq p < \infty$. Let $H : \Omega \times \Omega \rightarrow \mathbb{R}$ be Borel measurable on $\Omega \times \Omega$ such that for each $y \in \Omega$ the map $x \mapsto H(x,y)$ lies in $W^{k,p}(\Omega)$ and $H, D_x H ,D_x^2 H ,... D_x^k H \in L^p(\lambda \times \alpha)$ and all derivatives are all Borel measurable in $\Omega \times \Omega$. Then $A(x) := \int_\Omega H(x,y) d\alpha(y)$ lies in $W^{k,p}(\Omega)$.
Moreover one has
\begin{equation}\label{eq:346}
D^kA(x) = \int_\Omega D^kH(x,y)  \; \mathrm{d}\alpha(y) \quad k = 1,...,n \quad \mathrm{a.e.} \; x \in \Omega.
\end{equation}
\end{cor}
\begin{cor}[Sobolev Regularity of Minimizers] \label{cor:minregu}
Let $u\in \A(u_0)$ be a minimizer and $\beta \in (0,1]$. Then $u \in W^{3,2-\beta}(\Omega_{\epsilon_0}^C)$ for each $\beta > 0$ and the set of non-$1$-Lebesgue points of $D^2u$ in $\Omega_{\epsilon_0}^C$ has Hausdorff dimension $0$. Moreover, at every $1$-Lebesgue point of $D^2u$ which is not an atom of $\mu$  one has 
\begin{equation}\label{eq:347}
(D^2u)^*(x) =- \frac{1}{2} \int_\Omega D^2F(x,y)\; \mathrm{d}\mu(y) + D^2h(x), 
\end{equation}
where $F$, $\mu$ and $h$ are given in Lemma \ref{lem:bihammeasrep}.
\end{cor}
\begin{proof}
For the $W^{3,2-\beta}$-regularity we use the representation in Lemma \ref{lem:bihammeasrep} and Corollary  \ref{cor:regreg}. The requirements of Corollary \ref{cor:regreg} are satisfied if we can show that $F,D_x F,D_x^2F$ and $D_x^3F$ lie in $L^{2-\beta}(\lambda \times \mu)$ (since the remaining requirements follow immediately from Lemma \ref{lem:green}). We show this only for $D_x^3F$, the other computations are very similar. Using  \eqref{eq:thidbound}, Tonelli's Theorem and radial integration we find
\begin{align*}
\int_\Omega |D_x^3F&(x,y) |^{2-\beta} \; \mathrm{d}( \lambda \times \mu)(x,y)  = \int_\Omega \int_\Omega |D_x^3F(x,y)|^{2-\beta} \dx  \; \mathrm{d}\mu(y) \\ &  \leq \int_\Omega \int_\Omega \frac{C^{2-\beta}}{|x-y|^{2-\beta}} \dx \; \mathrm{d}\mu(y) 
 \leq C^{2-\beta} \int_\Omega \int_{B_{\diam(\Omega)}(y)} \frac{1}{|x-y|^{2-\beta}} \dx \; \mathrm{d}\mu(y) 
\\ & \leq C^{2-\beta} \int_\Omega \int_0^{\diam(\Omega)} 2\pi \frac{r}{r^{2-\beta}} \dr \; \mathrm{d}\mu(y) 
\leq C^{2-\beta} \int_\Omega 2\pi \diam(\Omega)^\beta \; \mathrm{d}\mu(y) \\ &  =2\pi C^{2-\beta} \diam(\Omega)^\beta \mu(\Omega) < \infty. 
\end{align*}
The $W^{3,2-\beta}$-regularity claim is shown. We conclude that $D^2u \in W^{1,2-\beta}(\Omega_{\epsilon_0}^C)$ for each $\beta > 0$. Since $\Omega_{\epsilon_0}^C$ has Lipschitz boundary, $D^2u$ extends to a function in $W^{1,2-\beta}(\mathbb{R}^n)$ (cf.  \cite[Thm.1, Sect.5.4]{Evans}). From \cite[Thm.1(i),(ii), Sect.4.8]{EvGar} follows that there is a Borel set $E_\beta \subset \Omega$ of $\beta$-Capacity zero, such that the non-$1-$Lebesgue points are contained in $E_\beta$. Now \cite[Thm.4, Sect.4.7]{EvGar} implies that $\mathcal{H}^{2\beta}(E_\beta)= 0$ and hence the set of non-$1-$Lebesgue points is a $\mathcal{H}^{2\beta}$ null set. Equation \eqref{eq:347} does not follow directly, since \eqref{eq:346} only gives one representative of $D^2u$. Let $x_0$ be a $1-$Lebesgue point of $D^2u$. Then, according to Lemma \ref{lem:green}
\begin{align*}
2 (\partial^2_{x_1x_1} u)^*(x_0)  & = 2\lim_{r \rightarrow 0 } \fint_{B_r(x_0)} (\partial^2_{x_1x_1} u)(y) \dy \\ &  = \lim_{r\rightarrow 0 } \fint_{B_r(x_0)} \int_\Omega \frac{-1}{8\pi} \left( 1 + 2 \frac{(x_1 - y_1)^2}{|x-y|^2} + 2 \log|x-y| \right) \dx  \; \mathrm{d}\mu(y) \\ & \quad \qquad + 2 \fint_{B_r(x_0)} \partial^2_{x_1x_1} h(x) \dx .
\end{align*}
Since $h$ is smooth, the last summand tends to $\partial_{x_1x_1}^2h(x_0)$. We have already shown above that $\partial_{x_1x_1}^2F = \frac{-1}{8\pi} \left( 1 + 2 \frac{(x_1 - y_1)^2}{|x-y|^2} + 2 \log|x-y| \right)$ lies in $L^{2-\beta}(\lambda \times \mu)$. Therefore we can interchange the order of the two integrations by Fubini's Theorem. Hence
\begin{align*}
2(\partial^2_{x_1x_1} u)^*(x_0) & = 2 \partial_{x_1x_1}^2 h(x_0)  \\ &+ \quad   \lim_{r\rightarrow0 }\int_\Omega \fint_{B_r(x_0)} \frac{-1}{8\pi} \left( 1 + 2 \frac{(x_1 - y_1)^2}{|x-y|^2} + 2 \log|x-y| \right) \dx \; \mathrm{d}\mu(y) .
\end{align*}
Now observe that  
\begin{equation*}
r \mapsto \fint_{B_r(x_0)} \frac{-1}{8\pi} \log|x-y| \dx  
\end{equation*}
is decreasing in $r$  because of Lemma \ref{lem:greensubham} and hence the monotone convergence theorem yields
\begin{equation} \label{eq:convi}
\lim_{r\rightarrow 0+} \int_\Omega \fint_{B_r(x_0)} \frac{-1}{8\pi}  \log|x-y| \dx \; \mathrm{d}\mu(y) = \int_\Omega \frac{-1}{8\pi}\log|x_0-y| \; \mathrm{d}\mu(y) . 
\end{equation}
(Actually, the monotone convergence theorem is not exactly applicable since the  integrand is not necessarily positive. This can however be fixed since $\mu$ is finite and for each $r$ the integrand is bounded from below by  $-\frac{1}{8\pi}\log \diam(\Omega)$. Adding and subtracting this quantity one obtains the claimed convergence).
Therefore 
\begin{align}\label{eq:lebptprf}
2(\partial^2_{x_1x_1} u)^*(x_0)&  =\lim_{r\rightarrow0 }\int_\Omega \fint_{B_r(x_0)} \frac{-1}{8\pi} \left( 1 + 2 \frac{(x_1 - y_1)^2}{|x-y|^2} \right) \dx \; \mathrm{d}\mu(y)\nonumber\\ & \qquad +  2\partial_{x_1x_1}^2 h(x_0)  - \int_\Omega \frac{1}{4\pi} \log|x_0-y| \dx \; \mathrm{d}\mu(y). 
\end{align}

Observe that  for $y \neq x_0$ one has 
\begin{equation*}
\lim_{r \rightarrow 0+} \fint_{B_r(x_0)} \frac{-1}{8\pi} \left( 1 + 2 \frac{(x_1 - y_1)^2}{|x-y|^2} \right) \dx = \frac{-1}{8\pi} \left( 1 + 2 \frac{((x_0)_1 - y_1)^2}{|x_0-y|^2} \right).
\end{equation*}
Since $\mu(\{x_0\})= 0$ the integrand converges $\mu$-almost everywhere to the right hand side. This and fact that the expression is uniformly bounded in $r$ by $\frac{3}{8\pi}$ imply together with the dominated convergence theorem that  
\begin{equation*}
\lim_{r\rightarrow0+}\int_\Omega \fint_{B_r(x_0)} \frac{-1}{8\pi} \left( 1 + 2 \frac{(x_1 - y_1)^2}{|x-y|^2} \right) \dx \; \mathrm{d}\mu(y) = \int_\Omega  \frac{-1}{8\pi} \left( 1 + 2 \frac{((x_0)_1 - y_1)^2}{|x_0-y|^2} \right) \; \mathrm{d}\mu(y).
\end{equation*}
Plugging this into \eqref{eq:lebptprf} we find 
\begin{equation*}
2(\partial^2_{x_1x_1} u)^*(x_0) = \int_\Omega  \frac{-1}{8\pi} \left( 1 + 2 \frac{((x_0)_1 - y_1)^2}{|x_0-y|^2} + 2 \log|x_0-y| \right)\; \mathrm{d}\mu(y) + 2 \partial^2_{x_1x_1} h(x_0) .
\end{equation*}
The same techniques apply for $(\partial_{x_1x_2}^2u)^*$ and $(\partial_{x_2x_2}^2u)^*$. This proves \eqref{eq:347}.
\end{proof}

\begin{cor}\label{cor:lebgem}
Let $u \in \A(u_0)$ be a minimizer. Then $\partial^2_{x_1x_2}u$ and $\partial^2_{x_1x_1} u - \partial^2_{x_2x_2}u $ lie in $L^\infty(\Omega_{\epsilon_0}^C)$. Moreover, each $x_0 \in \Omega$ that is not an atom of $\mu$ is a Lebesgue point of $\partial^2_{x_1x_2}u$ and $\partial^2_{x_1x_1} u - \partial^2_{x_2x_2}u $. 
\end{cor}
\begin{proof}
For the fact that $\partial^2_{x_1x_1} u - \partial^2_{x_2x_2} u  \in L^\infty(\Omega_{\epsilon_0}^c)$ observe with the notation of \eqref{eq:347} that  almost everywhere one has
\begin{align*}
|\partial^2_{x_1x_1} u - \partial^2_{x_2x_2} u | & = \left\vert- \frac{1}{2} \int_\Omega (\partial^2_{x_1x_1}F - \partial^2_{x_2x_2} F) \; \mathrm{d}\mu(y)  + \partial^2_{x_1x_1} h - \partial^2_{x_2x_2} h \right\vert  \\
& \leq \frac{3}{16 \pi} \mu(\Omega) + 2|| D^2h||_\infty < \infty,
\end{align*} 
where we used Lemma  \ref{lem:green} in the last step. Similarly one shows that $\partial^2_{x_1x_2} u  \in L^\infty(\Omega_{\epsilon_0}^C)$. 
Now we show that each non-atom $x$ of $\mu$ is a $1-$Lebesgue point of $\partial^2_{x_1x_2 }u $. By \eqref{eq:347} it is sufficient to show that each non-atom of $\mu$ is a $1-$Lebesgue point of $\int_\Omega \partial^2_{x_1x_2}F(\cdot,y) d\mu(y) $ as each point in $\Omega_{\epsilon_0}^C$ is a Lebesgue point of $D^2h$. We have already discussed in Corollary \ref{cor:minregu} that $\partial^2_{x_1x_2}F$ is $(\lambda \times \mu)$-measurable. Moreover it is product integrable as it is uniformly bounded.  

By Fubini's theorem
\begin{align*}
& \frac{1}{|B_r(x)|} \int_{B_r(x) } \left(  \int_\Omega  \partial_{x_1x_2}^2 F(z,y) \; \mathrm{d}\mu(y) \right) \; \mathrm{d}z  \\ & \quad \quad \quad \quad \quad \quad \quad \quad  =  \int_\Omega \left( \frac{1}{|B_r(x)|} \int_{B_r(x)} \partial^2_{x_1x_2} F(z,y)\; \mathrm{d}z \right) \; \mathrm{d}\mu(y) . 
\end{align*}
For each $y \in \Omega \setminus \{x\} $ the expression in parentheses converges to  $\partial^2_{x_1x_2} F(x,y)$ as $r \rightarrow 0 $ and since $x$ is not an atom of $\mu$ the expression converges to $ \partial^2_{x_1x_2} F(x,y)$ $\mu$-almost everywhere. Moreover Lemma \ref{lem:green} yields that the expression is uniformly bounded by $\frac{3}{8\pi}$ 
 and hence the dominated convergence theorem yields 
 \begin{equation}\label{eq:3.26}
 \lim_{r\rightarrow 0 } \frac{1}{|B_r(x)|} \int_{B_r(x) } \left(  \int_\Omega  \partial_{x_1x_2}^2 F(z,y) \; \mathrm{d}\mu(y) \right) \; \mathrm{d}z  = \int_\Omega \partial^2_{x_1x_2} F(x,y) \; \mathrm{d}\mu(y).
 \end{equation}
 To show the Lebesgue point property it remains to show that 
 \begin{equation*}
 \lim_{r \rightarrow 0} \frac{1}{|B_r(x)|} \int_{B_r(x)}  \left\vert \int_\Omega \partial^2_{x_1x_2} F(z,y) \; \mathrm{d}\mu(y) -  \int_\Omega \partial^2_{x_1x_2} F(x,y) \; \mathrm{d}\mu(y) \right\vert = 0 .
 \end{equation*}
 This is immediate once one observes with the triangle inequality and Fubini's theorem that 
 \begin{align*}
 & \frac{1}{|B_r(x)|} \int_{B_r(x)}  \left\vert \int_\Omega \partial^2_{x_1x_2} F(z,y) \; \mathrm{d}\mu(y) -  \int_\Omega \partial^2_{x_1x_2} F(x,y) \; \mathrm{d}\mu(y)  \right\vert \; \mathrm{d}z
 \\ &  \quad \quad \quad \quad \quad \quad \leq \int_\Omega \frac{1}{|B_r(x)|}\int_{B_r(x)} |\partial^2_{x_1x_2} F(z,y) - \partial^2_{x_1x_2} F(x,y)| \; \mathrm{d}z \; \mathrm{d}\mu(y) .
 \end{align*}
 The term on the right hand side can be shown to tend to zero as $r \rightarrow 0 $ with the dominated convergence theorem using arguments similar to the discussion before \eqref{eq:3.26}. For $\partial^2_{x_1x_1} u - \partial^2_{x_2x_2}u $ the analogous statement can be shown similarly.  
\end{proof}

\begin{cor}\label{cor:suubhaam}
Let $u \in \A(u_0)$ be a minimizer. Then for each $x \in \Omega$ the quantity $(\Delta u)^*(x) := \lim_{r \rightarrow 0 } \avint_{B_r(x)} \Delta u(y) dy$ exists in $[0, \infty]$. Moreover, the map $x \mapsto (\Delta u)^*(x)$ is superharmonic.  
\end{cor}
\begin{proof}
Recall that by $\Delta u $ is weakly superharmonic by \eqref{eq:subbiham}. By \cite[Theorem 4.1]{Serrin} follows immediately that $(\Delta u)^*(x)$ exists in $\mathbb{R} \cup \{ \infty \}$ for all $x \in \Omega$. By Corollary \ref{lem:corsubham} it has to lie in $[0, \infty]$, which shows the first part of the claim. From \eqref{eq:347} and Lemma \ref{lem:green} we infer that 
\begin{equation*}
\Delta u (x) = - \frac{1}{4\pi} \int_\Omega (\log|x-y| + 1 ) \; \mathrm{d}\mu(y) + \Delta h (x) \quad a.e. .
\end{equation*}
Similar to the discussion in \eqref{eq:convi} we can derive, using the special properties of the logarithm that 
\begin{equation}\label{eq:3.32}
(\Delta u)^*(x) = - \frac{1}{4\pi} \int_\Omega (\log|x-y| + 1 ) \; \mathrm{d}\mu(y) + \Delta h (x) \quad \forall x \in \Omega_{\epsilon_0}^C.
\end{equation}
Note that $(\Delta u)^*$ is the so-called \emph{canonical representative} of a weakly subharmonic function in the sense of \cite[p.360]{Serrin}. To show that $(\Delta u)^*$ is subharmonic it suffices according to \cite[Theorem 4.3]{Serrin} to show that $(\Delta u)^*$ is lower semicontinuous. For this let $(x_n)_{n = 1}^\infty  \subset \Omega_{\epsilon_0}^C$ be such that $x_n \rightarrow x \in \Omega_{\epsilon_0}^C$. Note that $ - \log|x_n - \cdot |$ is bounded from below independently of $n$ by $-\log\diam(\Omega)$. Thus Fatou's lemma yields 
\begin{equation}\label{eq:faatou}
\liminf_{n \rightarrow \infty} \int_\Omega - \log |x_n - y | \; \mathrm{d}\mu(y) \geq \int_\Omega \liminf_{n\rightarrow \infty} ( - \log |x_n - y| \; \mathrm{d}\mu (y) = -\int_\Omega  \log |x-y| \; \mathrm{d}\mu (y) .
\end{equation}
Since $(\Delta u)^*(x_n)$ consists only of continuous terms and a positive multiple of the left hand side in \eqref{eq:faatou}, one has $\liminf_{n \rightarrow \infty } (\Delta u)^*(x_n) \geq (\Delta u)^*(x)$, that is $(\Delta u)^*$ is lower semicontinuous.  As we already explained this implies superharmonicity of $(\Delta u)^*$. 
\end{proof}
\begin{remark}
Note that the notation $(\Delta u)^*$ creates a slight ambiguity with \eqref{eq:avliim}, namely whenever the limit in the definition is infinite. It will always be clear from the context what convention is used, especially in view of the following consistency result.
\end{remark}
\begin{prop}\label{prop:genfac}
Let $f : \Omega \rightarrow \mathbb{R}$ be  a  nonnegative superharmonic function. Then each point where $f< \infty$ is a $1-$Lebesgue point of $f$. 
\end{prop}
\begin{proof}
By \cite[Theorem 3.1.3]{Armitage} one has $f(x) = \liminf_{y \rightarrow x} f(y)$ for each $x \in \Omega$. In particular 
\begin{equation}\label{eq:liminf}
f(x) = \lim_{r \rightarrow 0 } \inf_{B_r(x)} f . 
\end{equation}
Now suppose that $f(x)< \infty$. Then by the triangle inequality
\begin{align*}
\fint_{B_r(x)} |f(z) - f(x) |  \; \mathrm{d}z & \leq \fint_{B_r(x)} | f(z) - \inf_{B_r(x)} f|  \; \mathrm{d}z  +   | \inf_{B_r(x)} f - f(x) | 
\\ & \leq \fint_{B_r(x)} f(z) \; \mathrm{d}z - \inf_{B_r(x)} f + | \inf_{B_r(x)} f - f(x) |.
\end{align*}
As $f$ is superharmonic we have 
$
\fint_{B_r(x)} f(z) dz \rightarrow f(x)$ as $r \rightarrow 0+$.
Using this, $f(x)< \infty$ and \eqref{eq:liminf} we obtain that 
\begin{equation*}
 \lim_{r \rightarrow 0 } \fint_{B_r(x)} |f(z) - f(x) | \; \mathrm{d}z = f(x) - f(x) = 0 . \qedhere
\end{equation*}
\end{proof}

Putting the previous results together we obtain the following 

\begin{cor}
Each non-$1-$Lebesgue point $x$ of $D^2u $ is an atom of $\mu$ or satisfies $(\Delta u)^*(x) = \infty$. 
\end{cor}
\begin{proof}
Suppose that $x$ is neither an atom of $\mu$ nor $(\Delta u)^*(x) = \infty$. By Corollary \ref{cor:suubhaam} and Proposition \ref{prop:genfac} we get that $x$ is a $1-$Lebesgue point of $\Delta u$. By Corollary \ref{cor:lebgem} we also know that $x$ is a $1-$Lebesgue point of $\partial^2_{x_1x_1} u - \partial^2_{x_2x_2} u $ and $\partial^2_{x_1x_2} u $. Since all second derivatives of $u$ are linear combinations of the mentioned quantities, $x$ is a $1-$Lebesgue point of $D^2u$. The claim follows by contraposition. 
\end{proof}

We can refine the statement with the following observations

\begin{lemma}
Let $u \in \A(u_0)$ be a minimizer. If $x_0 \in \Omega$ is an atom of $\mu$ then  $(\Delta u)^*(x_0) = \infty$. 
\end{lemma}

\begin{proof}
Suppose that $x_0 $ is an atom of $\mu$ and set $\widetilde{\mu} := \mu - \mu(\{x_0\}) \delta_{x_0}$ which is also a finite measure. 
Using \eqref{eq:3.32} we find with the notation from there that for each $x \in \Omega_{\epsilon_0}^C$ 
\begin{align*}
(\Delta u)^*(x) & = - \frac{1}{4\pi} \int_\Omega (\log |x-y| + 1 )\; \mathrm{d}\mu(y) + \Delta h(x) \\ & = - \frac{1}{4\pi} (\log |x-x_0| + 1) \mu(\{x_0\}) - \frac{1}{4\pi}\int_\Omega (\log |x-y| + 1 ) \; \mathrm{d}\widetilde{\mu}(y) + \Delta h (x) 
\\ & \geq - ||\Delta h ||_\infty - \frac{1}{4\pi}( \log \diam(\Omega) + 1) \widetilde{\mu}(\Omega) - \frac{\mu(\{ x_0\} )}{4\pi} ( 1 + \log |x-x_0| ) . 
\end{align*}
Plugging in $x= x_0 $ we obtain finally that $(\Delta u)^*(x_0) = \infty$ as claimed. 
\end{proof}

\begin{remark}\label{rem:atoom}
The previous observations show that each non-$1-$Lebesgue point of $D^2u$ satisfies $(\Delta u)^* = \infty$ and each atom of $\mu$ is a non-$1-$Lebesgue point of $D^2u$. 
\end{remark}

\begin{lemma}[Semiconvexity]\label{lem:semicon} 
Let $u \in \A(u_0)$ be a minimizer and set 
\begin{equation*}
A := \frac{\sqrt{5}}{2} \left(2 ||D^2h||_\infty + \frac{3}{16\pi} \mu(\Omega) \right).
\end{equation*}
Then at each $x \in \Omega_{\epsilon_0}^C$ which is $1-$Lebesgue point of $D^2u$  the matrix $(D^2 u)^* + AI $ is positive semidefinite, where $I = \mathrm{diag}(1,1)$ denotes the identity matrix. In particular, for each $x_0 \in \mathbb{R}^2$ one has  that $x \mapsto u(x) + \frac{1}{2} A |x-x_0|^2$ is convex on $\Omega_{\epsilon_0}^C$. 
\end{lemma}
\begin{proof}
Let $x$ be a Lebesgue point of $D^2u$.  By Remark \ref{rem:atoom}, $x$ is not an atom of $\mu$. 
Note that if $M  = \begin{pmatrix} m_{11} & m_{12} \\ m_{12} & m_{22} \end{pmatrix} \in \mathbb{R}^{2\times 2} $ is a symmetric matrix then the eigenvalues of $M$ are given by 
\begin{equation}\label{eq:eigw}
\lambda_{1,2} = \frac{m_{11} + m_{22}}{2} \pm \sqrt{\frac{1}{4} (m_{11} - m_{22} )^2 + m_{12}^2 }.
\end{equation}
If $M = (D^2u)^*(x)  + AI$ then Corollary \ref{lem:corsubham} implies that 
\begin{equation}\label{eq:eigenval}
\frac{m_{11} + m_{22}}{2} = (\Delta u)^* + 2A \geq 2 A . 
\end{equation}
Using \eqref{eq:347}, the fact that $x$ is not an atom of $\mu$, and Lemma \ref{lem:green} we obtain 
\begin{align*}
|m_{11}- m_{22}|& = \left\vert\frac{1}{2} \int_\Omega (\partial^2_{x_1x_1} F(x,y) - \partial^2_{x_2x_2}F(x,y) ) \; \mathrm{d}\mu(y) + \partial^2_{x_1x_1}h - \partial_{x_2x_2}^2h \right\vert \\ & \leq\frac{1}{2} \int_{\Omega \setminus \{x\} } |\partial^2_{x_1x_1} F(x,y) - \partial^2_{x_2x_2}F(x,y) | \; \mathrm{d}\mu(y) + 2 ||D^2h||_\infty\\ &  \leq \frac{3}{16\pi}\mu(\Omega) + 2||D^2h||_\infty
. 
\end{align*} 
Analogously one can show that 
\begin{equation*}
|m_{12}| \leq \frac{3}{16\pi} \mu(\Omega) + 2 ||D^2h ||_\infty.
\end{equation*}
Hence
\begin{equation*}
\sqrt{\frac{1}{4} (m_{11} - m_{22} )^2 + m_{12}^2 } \leq \sqrt{ \left( 1 + \frac{1}{4}\right) \left(\frac{3}{16\pi} \mu(\Omega) + 2 ||D^2h ||_\infty\right)^2 } \leq A.
\end{equation*}
Plugging this and \eqref{eq:eigenval} into \eqref{eq:eigw} we find
\begin{equation*}
\lambda_{1,2} \geq 2 A - A = A \geq 0. 
\end{equation*}
Thus we obtain that $M$ is indeed positive semidefinite. For $\epsilon >0 $ let $\rho_\epsilon$ be the standard mollifier. Set $f_\epsilon(x) := \left(  u(\cdot) + \frac{1}{2}A |\cdot - x_0|^2\right) * \rho_\epsilon$. Observe that for $\epsilon< \epsilon_0$,  $f_\epsilon \in C^2(\Omega_{\epsilon_0}^C)$ and $D^2 f_\epsilon = (D^2 u + AI) * \rho_\epsilon$ on $\Omega_{\epsilon_0}^C$. This matrix is positive semidefinite since for each $z \in \mathbb{R}^2$
\begin{equation*}
z^TD^2f_\epsilon(x) z = (z^T(D^2u + AI) z * \rho_\epsilon)(x) \geq 0 ,  
\end{equation*}
as $\rho_\epsilon $ is nonnegative and $D^2 u + AI$ is positive semidefinite almost everywhere. Hence $f_\epsilon$ is convex. However $f_\epsilon$ also converges to $u + \frac{1}{2} A |\cdot - x_0|^2$ uniformly on $\Omega_{\epsilon_0}^C$ as the latter function is continuous. It is easy to verify with the definition of convexity that uniform limits of convex functions are convex again. 
\end{proof}
 \section{Emptyness of The Singular Nodal Set}
 In this section we study the gradient $\nabla u $ at points where $u$ vanishes. Whenever we refer to the gradient, we always mean its continuous representative, cf. Remark \ref{rem:ceins}. 
We show that the set  $\{ u = \nabla u = 0 \}$, which we refer to as \emph{singular nodal set}, is empty. It is vital for the argument to look at the behavior of the Hessian at points that lie in the singular nodal set. We have to distinguish between $1-$Lebesgue points of the Hessian and non-$1-$Lebesgue points of the Hessian. The $1-$Lebegue points can be discussed using blow-up arguments. For non $1-$Lebesgue points, one will profit from the characterization in Remark \ref{rem:atoom}.
 
 The blow-up arguments in this section are based on the following version of Aleksandrov's theorem, which allows for a second order Taylor-type expansion.  
 \begin{lemma}[A version of Aleksandrovs theorem in $\mathbb{R}^n$]\label{lem:aleks}\label{rem:quasialeks}
 Let $\Omega \subset \mathbb{R}^n$ be bounded and $f \in W^{2,2}(\Omega)\cap C^1(\Omega)$ be $A$-semiconvex for some $A \in \mathbb{R}$. If $x_0\in \Omega$ is a $1-$Lebesgue point of $D^2f$, then 
 \begin{equation}\label{eq:aleks}
 f(x) - f(x_0) - \nabla f(x_0) (x-x_0) - \frac{1}{2}(x-x_0)^T (D^2f)^*(x_0) (x-x_0) = o(|x-x_0|^2 ).
\end{equation}  
 \end{lemma} 
 \begin{proof}
By considering $\widetilde{f} := f +  \frac{1}{2}A|\cdot-x_0|^2$ we can assume without loss of generality that $f$ is convex. Note that for convex functions \cite[Thm.2,Sect.6.3]{EvGar} yields that $D^2f = (\mu_{i,j})_{i,j=1,...,n}$ for signed Radon measures $\mu_{i,j}$ in the sense of distributions. Hence one can also decompose the measures in their absolutely continuous and singular parts, i.e.  $D^2f = [D^2f]_{ac} + [D^2f]_s$. In our case $[D^2f]_s = 0$ because of the additional regularity assumption that $f \in W^{2,2}(\Omega)$. Moreover $[D^2f]_{ac} = D^2 f \cdot \lambda$, where $\lambda$ denotes the $n$-dimensional Lebesgue measure. 
In  \cite[Thm.1,Sect.6.4]{EvGar}, a proof of the classical Aleksandrov theorem is given, and examining Part $1$ of the given proof, it is shown that \eqref{eq:aleks} holds for each convex $f$ and each point $x_0$ such that
\begin{enumerate}
\item $\nabla f(x_0)$ exists and $x_0$ is a $1-$Lebesgue point of $\nabla f$.  
\item $x_0$ is a $1-$Lebesgue point of the Radon-Nikodym density of $[D^2f]_{ac}$ 
\item $x_0$ satisfies $\lim_{r\rightarrow 0 } \frac{1}{r^n} [D^2f]_s(B_r(x_0)) = 0 $
\end{enumerate}
Since $f$ was assumed to be $C^1$, each point $x_0$ trivially satisfies $(1)$. As we mentioned above $[D^2f]_s = 0$, so each point $x_0$ automatically satisfies $(3)$. Hence the proof works for each point $x_0$ satisfying $(2)$, i.e. each $1-$Lebesgue point of $D^2f$. \qedhere
 \end{proof}
 \begin{remark}\label{rem:tayli}
For $f$ as in the statement of Lemma \ref{lem:aleks}, Equation \eqref{eq:aleks} can be seen as a Taylor expansion around each $1$-Lebesgue point of $D^2f$. In particular note that each $1-$Lebesgue point $x_0$ of $D^2f$ such that $\nabla f (x_0) = 0 $ and $(D^2f)^*(x_0)$ is positive definite is a strict local minimum of $f$. 
 \end{remark}
 
 \begin{lemma}[Hessian on Singular Nodal Set- I]
 Let $u \in \A(u_0)$ be a minimizer and $x_0 \in \Omega_{\epsilon_0}^C$ be a Lebesgue point of $D^2u$  such that $u(x_0) = \nabla u(x_0) = 0$. Then either $(D^2u)^*(x_0)= 0 $ or $(D^2 u)^*(x_0)$ is positive definite. 
 \end{lemma}
 \begin{proof}
 First define 
 \begin{equation*}
 C := \limsup_{\epsilon \rightarrow 0+ } \frac{|\{ 0 < u < \epsilon \}| }{\epsilon},
 \end{equation*}
which is finite because of Corollary \ref{cor:posdir}. By Lemma \ref{lem:semicon} and Lemma \ref{rem:quasialeks} we get the following blow-up profile at $x_0$: 
 \begin{equation}\label{eq:425}
 \frac{u(x_0 + \sqrt{\epsilon} w)}{\epsilon} \rightarrow \frac{1}{2} w^T (D^2u)^*(x_0) w ,
 \end{equation}
 locally uniformly as $\epsilon \rightarrow 0$. 
  Now fix $ \tau > 0$  and observe that 
  \begin{equation*}
  C \geq \limsup_{\epsilon \rightarrow 0+ } \frac{|\{ 0 < u <  \epsilon \} \cap B_{\tau \sqrt{\epsilon}}(x_0) | }{ \epsilon}.
  \end{equation*}
 Using scaling properties of the Lebesgue measure and  \eqref{eq:425} we get 
  \begin{align*}
  C & \geq  \limsup_{\epsilon \rightarrow 0+ }  \frac{|\{ 0 < u < \epsilon \} \cap B_{\tau \sqrt{\epsilon}}(x_0) | }{ \epsilon}
  \\ & = \limsup_{\epsilon \rightarrow 0 } \frac{1}{\epsilon} \left\vert \left\lbrace \sqrt{\epsilon}w : w \in B_{\tau}(0)  \; \textrm{s.t.} \; 0 < \frac{u(x_0 + \sqrt{\epsilon} w)}{\epsilon} < 1 \right\rbrace \right\vert 
  \\ & = \limsup_{\epsilon \rightarrow 0 }  \left\vert \left\lbrace w \in B_{\tau}(0) :  0 < \frac{u(x_0 + \sqrt{\epsilon} w)}{\epsilon} < 1 \right\rbrace \right\vert 
  \\ & \geq \left\vert \left\lbrace w \in B_{\tau}(0) :  \frac{1}{4} < \frac{1}{2} w^T (D^2u)^*(x_0) w < \frac{1}{2} \right\rbrace \right\vert ,
  \end{align*}
  where the last step can be carried out because the convergence in \eqref{eq:425} is  uniform in $B_\tau(0)$.
  Now let $\lambda_1 , \lambda_2 $ be the eigenvalues of $\frac{1}{2}(D^2u)^*(x_0)$. Since $(D^2u)^*$ is symmetric, we can use an orthogonal transformation to obtain 
  \begin{equation*}
  C \geq \left\vert \left\lbrace w \in B_{\tau}(0) :  \frac{1}{4} < \lambda_1 w_1^2 + \lambda_2 w_2^2 < \frac{1}{2} \right\rbrace \right\vert.
  \end{equation*}
  since $\tau > 0 $ was arbitrary we can let $\tau \rightarrow \infty$ to find 
  \begin{equation}\label{eq:measres}
  C \geq \left\vert \left\lbrace w \in \mathbb{R}^2 :  \frac{1}{4} < \lambda_1 w_1^2 + \lambda_2 w_2^2 < \frac{1}{2} \right\rbrace \right\vert.
  \end{equation}
  Recall moreover from Corollary \ref{lem:corsubham} that
  \begin{equation}\label{eq:laplgro}
  0 \leq (\Delta u)^* =\mathrm{tr}((D^2u)^*) = 2( \lambda_1 + \lambda_2) .
  \end{equation}
  Now we distinguish cases to show that $\lambda_1 = \lambda_2 = 0$ or $\lambda_1,\lambda_2 > 0$. Assume that none of the two cases apply. One out of the two eigenvalues has to be positive because of \eqref{eq:laplgro} and the other one has to be zero or negative. Without loss of generality $\lambda_1 > 0 $.
If $\lambda_2$  is negative one can observe that if $w_1 > 0$  
\begin{equation*}
\frac{1}{4} < \lambda_1 w_1^2 + \lambda_2 w_2^2 < \frac{1}{2} \quad \Leftrightarrow  \frac{1}{\sqrt{\lambda_1}}\sqrt{\frac{1}{4} + |\lambda_2| w_2^2} < w_1 < \frac{1}{\sqrt{\lambda_1}}\sqrt{\frac{1}{2} + |\lambda_2| w_2^2}.
\end{equation*}
  Therefore \eqref{eq:measres} yields, using that for positive $a,b$ one has  $a-b = \frac{a^2-b^2}{a+b}$ one has
  \begin{align*}
C & \geq \int_0^\infty \frac{1}{\sqrt{\lambda_1}} \left(  \sqrt{\frac{1}{2} + |\lambda_2| w_2^2} - \sqrt{\frac{1}{4} + |\lambda_2| w_2^2} \right) \; \mathrm{d}w_2 \\ & =  \frac{1}{\sqrt{\lambda_1}} \int_0^\infty \frac{1}{4} \frac{1}{ \sqrt{\frac{1}{2} + |\lambda_2| w_2^2} + \sqrt{\frac{1}{4} + |\lambda_2| w_2^2} } \; \mathrm{d}w_2  \\ & \geq \frac{1}{\sqrt{\lambda_1}} \int_0^\infty \frac{1}{8 } \frac{1}{\sqrt{\frac{1}{4} + |\lambda_2| w_2^2 }}dw_2 = \infty,
  \end{align*}
  a contradiction. If $\lambda_1 > 0 $ and $\lambda_2 = 0 $ then it is easy to see that the right hand side of \eqref{eq:measres} equals infinity again. Therefore we obtain a contradiction and hence $\lambda_1 = \lambda_2 = 0 $ or $\lambda_1,\lambda_2 >0$. Since $(D^2u)^*(x_0)$ is symmetric and therefore diagonalizable we obtain that $(D^2u)^*(x_0) = 0 $ or $(D^2u)^*(x_0)$ is positive definite.
 \end{proof}
 

Next we exclude that $(D^2u)^*(x_0)$ is positive definite using a variational argument. 

\begin{lemma}[Hessian on Singular Nodal Set - II] \label{lem:Hesssing}
 Let $u \in \A(u_0)$ be a minimizer and $x_0 \in \Omega_{\epsilon_0}^C$ be a Lebesgue point of $D^2u$  such that $u(x_0) = \nabla u(x_0) = 0$. Then $(D^2u)^*(x_0)= 0 $.
\end{lemma} 
\begin{proof}
By the previous lemma, it remains to show that $(D^2u)^*(x_0)$ is not positive definite. To do so, we suppose the opposite, i.e. $(D^2u)^*(x_0)$ is positive definite. By   Lemma \ref{rem:quasialeks} and Remark \ref{rem:tayli} $x_0$ is a strict local minimum of $u$ and grows quadratically away from $x_0$, i.e. there exists $r_0> 0 $ and $\beta > 0 $ such that $0 < u(x) < \beta |x-x_0|^2 $ for each $x \in B_{r_0}(x_0) \setminus \{x_0 \}$. Let $r \in (0, r_0)$ be arbitrary. Now choose $\phi \in C_0^\infty(B_r(x_0))$ such that  $ 0 \geq \psi \geq - 1$ and $\psi \equiv - 1$ in $B_\frac{r}{2}(x_0)$. As for each $\epsilon > 0 $ the function $u + \epsilon \psi $ is admissible, one has 
\begin{align*}
\E(u) & \leq \E(u + \epsilon \psi)  \leq \int_\Omega ( \Delta u )^2 \; \mathrm{d}x + 2 \epsilon \int_\Omega \Delta u \Delta \psi \; \mathrm{d}x \\ & \quad  + \epsilon^2 \int_\Omega (\Delta \psi)^2 \; \mathrm{d}x + \{ u > 0 \} | - |\{ x \in B_\frac{r}{2}(x_0) :0 < u(x) < \epsilon \} | \\
& = \E(u) - \epsilon \int_\Omega \psi \; \mathrm{d}\mu  + \epsilon^2 \int_\Omega (\Delta \psi)^2 - |\{ x \in B_\frac{r}{2}(x_0) : u(x) < \epsilon \} |,
 \end{align*}
 where we used \eqref{eq:bihammeas} and the strict local minimum property of $x_0$  in the last step. 
 Note that 
\begin{equation*} 
 | \{ x \in B_\frac{r}{2}(x_0) : u(x) < \epsilon \} | \geq | \{  x \in B _\frac{r}{2}(x_0) : \beta |x-x_0|^2 < \epsilon \} | = \min \left\langle \pi \frac{\epsilon}{\beta}, \frac{r}{2} \right\rangle. 
 \end{equation*} 
 We can compute for each $\epsilon < \frac{\beta r}{2\pi}$  
 \begin{align*}
 \E(u)  & \leq \E(u) - \epsilon \int_\Omega \psi d\mu + \epsilon^2 \int_\Omega (\Delta \psi)^2 - \epsilon \frac{\pi}{\beta}
 \\ & \leq \E(u) + \epsilon  \left(\mu(B_r(x_0)) -  \frac{\pi}{\beta} \right) + \epsilon^2 \int_\Omega (\Delta \psi)^2.
 \end{align*}
 Rearranging and dividing by $\epsilon$ we obtain 
 \begin{equation*}
 - \mu(B_r(x_0)) + \frac{\pi}{\beta} \leq \epsilon \int_\Omega (\Delta \psi)^2
 \end{equation*}
 Letting first $\epsilon \rightarrow 0 $ and then $r \rightarrow 0 $ we find 
 \begin{equation*}
 \frac{1}{\beta} \leq \mu(\{x_0\}) = 0,
 \end{equation*}
 where we used in the last step that by Remark \ref{rem:atoom} $x_0$ is not an atom of $\mu$. Finally, we obtain a contradiction. 
\end{proof}

\begin{lemma}[Hessian on Singular Nodal Set - III] \label{lem:zerohaus}
Let $u\in \A(u_0)$ be a minimizer. Then $\{ u = \nabla u = 0 \}$ does not contain any $1-$Lebesgue points of $D^2u$. In particular $\{u = \nabla u = 0 \}$ is of zero Hausdorff dimension and each $x_0 \in \{ u = \nabla u = 0 \} $ satisfies $(\Delta u)^*(x_0) = \infty$. 
\end{lemma}
\begin{proof}
Assume that $\{ u = \nabla u = 0 \}$ contains a Lebesgue point $x_0$ of $D^2 u$. Then, according to the previous Lemma, $(D^2u)^*(x_0) = 0$. This implies in particular that $(\Delta u)^*(x_0) = 0 $. Now note that by Corollary \ref{cor:suubhaam} $(\Delta u)^*$ is a nonnegative superharmonic function. Nonnegativity of $(\Delta u)^*$  implies that $x_0$ is a point where $(\Delta u)^*$ attains its global minimum in $\Omega$, namely zero. By the strong maximum principle it follows that $(\Delta u)^* \equiv 0 $, which would however imply that $u$ is harmonic and hence positive since its boundary data $(u_0)_{\mid_{\partial \Omega}}$ are strictly positive. Thus $\{ u = \nabla u = 0 \}= \emptyset$, contradicting the existence of $x_0$. The  first sentence of the statement follows.  The second sentence of the statement follows immediately from Corollary \ref{cor:minregu} and  Remark \ref{rem:atoom}. 
\end{proof}

 \begin{lemma}[Singular Nodal Points are Isolated]
 Suppose that $x_0 \in \{ u = \nabla u =  0 \}$. Then there exists $r > 0 $ such that $u$ is convex and nonnegative on $B_r(x_0)$. Moreover, $B_r(x_0) \cap \{ u = 0 \} = \{ x_0 \} $.   
 \end{lemma}
 \begin{proof} First we show convexity. As an intermediate step we show that there exists $r  >0 $ such that for each $1-$Lebesgue point $x$ of $D^2u $ in $B_r(x_0)$ the matrix $(D^2u)^*(x)$ is positive definite. Note that by Corollary \ref{cor:lebgem} there exists $M > 0 $ such that for each $1-$Lebesgue point $x$ of $D^2 u$  in $\Omega_{\epsilon_0}^C$ one has
 \begin{equation}\label{eq:difeigenval}
 | ( \partial^2_{x_1x_1}u )^* - ( \partial^2_{x_2x_2} u )^* | \leq M
 \end{equation}
 and 
 \begin{equation*}
 |( \partial^2_{x_1x_2} u )^*| \leq M 
 \end{equation*}
 As $(\Delta u )^*$ is subharmonic by Corollary \ref{cor:suubhaam}, \cite[Theorem 3.1.3]{Armitage} yields that $(\Delta u)^*(x_0) = \liminf_{x \rightarrow x_0} (\Delta u)^*(x)$, which equals infinity by the previous lemma. Hence one can find $\overline{r} > 0 $ such that $(\Delta u)^* > 5 M $ on $B_{\overline{r}}(x_0)$. If $x \in B_{\overline{r}}(x_0)$ is now a $1-$Lebesgue point of $D^2 u$ this implies that   $(\partial^2_{x_1x_1}u )^*(x) + ( \partial^2_{x_2x_2} u )^*(x)  \geq 5M$. Together with \eqref{eq:difeigenval} we obtain that $(\partial^2_{x_ix_i} u)^*(x) \geq 2M $ for all $i = 1,2$. Now we can show using the principal minor criterion that $(D^2u)^*(x)$ is positive definite. Indeed $(\partial^2_{x_1x_1} u)^*(x) \geq 2M > 0$ and $\det (D^2u)^*(x) = (\partial^2_{x_1x_1} u)^*(x)(\partial^2_{x_2x_2} u)^*(x) -(\partial^2_{x_1x_2} u)^*(x)^2 \geq 4M^2 - M^2 > 0 $. All in all, $(D^2u)^*$ is positive definite on $B_{\overline{r}}(x_0)$. We will show next that this implies convexity of $u$ on a smaller ball. For $\epsilon  \in (0, \frac{\overline{r}}{2})$ let $\phi_\epsilon$ be the standard mollifier with support in $B_\epsilon(0)$. Note  that  $D^2(u * \phi_\epsilon) = D^2u * \phi_\epsilon $ on $B_\frac{\overline{r}}{2}(x_0)$. As an easy computation shows,  $(D^2u * \phi_\epsilon)(x)$ is positive definite for each $x \in B_\frac{\overline{r}}{2}(x_0)$. Therefore $u* \phi_\epsilon $ is convex on $B_\frac{\overline{r}}{2}(x_0)$. Eventually, $u$ is convex on $B_\frac{\overline{r}}{2}(x_0)$ as uniform limit of convex functions. Choosing $r := \frac{\overline{r}}{2}$ implies the desired convexity. Convexity also implies that for each $x,y \in B_r(x_0)$ one has 
 \begin{equation*}
 u(x) - u(y) \geq \nabla u(y) \cdot (x-y) .
 \end{equation*}
 Plugging in $y = x_0$, we obtain $u(x) \geq 0 $ which shows the desired nonnegativity on $B_r(x_0)$. It remains to  show that $B_r(x_0) \cap \{ u = 0 \} = \{ x_0 \}$. Assume that there is a point $x_1 \in B_r(x_0)$ such that $u(x_1)=0$. By convexity and nonnegativity we obtain for each $\lambda \in (0,1)$ that 
 \begin{equation*}
 0 \leq u(\lambda x_1 + (1-\lambda) x_0 ) \leq \lambda u(x_1) + (1- \lambda) u(x_0) = 0 .
 \end{equation*}
Hence $u_{\mid_{\overline{x_0x_1}}}  \equiv 0 $, where $\overline{x_0x_1}$ denotes the line segment connecting $x_0$ and $x_1$. Now this line segment lies completely in $B_r(x_0)$ and because of the nonnegativity, each point in $\overline{x_0x_1}$ is a local minimum of $u$. This yields that $\nabla u $ vanishes on this line segment and hence $\overline{x_0x_1} \subset \{ u = \nabla u = 0 \}$. This contradicts  Lemma \ref{lem:zerohaus}, as $\{u = \nabla u = 0 \} $ must have zero Hausdorff dimension. The claim follows. 
 \end{proof}
 \begin{cor}(Emptyness of the Singular Nodal Set) \label{cor:nosing}
 Let $u \in \A(u_0) $ be a minimizer. Then $\{ u =  \nabla u = 0 \} = \emptyset $. 
 \end{cor}
 \begin{proof}
 Suppose that there exists some $x_0 \in \{ u = \nabla u = 0 \}$. Recall that then $(\Delta u)^*(x_0) = \infty$ by Lemma \ref{lem:zerohaus}. Also, by the previous Lemma, there exists $r > 0$ such that $\{u = 0 \} \cap B_r(x_0)= \{ x_0 \}$ and $u(x) > 0 $ for each $x \in B_r(x_0)\setminus \{ x_0 \}$. By possibly choosing a smaller radius $r$ we can achieve that $B_r(x_0) \subset \Omega_{\epsilon_0}^C$.  Now define $g_1 := u_{\mid_{ \partial B_\frac{r}{2}(x_0)}}$ and $g_2 := \nabla u_{\mid_{ \partial B_\frac{r}{2}(x_0)}}$. Note that $g_1,g_2 \in C^\infty( \partial B_\frac{r}{2}(x_0))$ by Lemma \ref{lem:biham}. By \cite[Theorem 2.19]{Sweers} one obtains that there exists a unique solution $h \in C^\infty (\overline{B_\frac{r}{2}(x_0)})$ such that 
 \begin{equation*}
 \begin{cases}
 \Delta^2 h = 0 & \textrm{in } B_\frac{r}{2}(x_0), \\ 
 h  = g_1 & \textrm{on } \partial B_\frac{r}{2}(x_0), \\
  \nabla h = g_2 & \textrm{in } \partial B_\frac{r}{2}(x_0). 
 \end{cases}
\end{equation*}
Moreover, as a standard variational argument shows, $h$ is uniquely determined by 
\begin{align*}
\int_{B_\frac{r}{2}(x_0) } (\Delta h)^2 \; \mathrm{d}x = \inf \Bigg\lbrace \int_{B_\frac{r}{2}(x_0) } (\Delta w)^2 \; \mathrm{d}x :  w \in & W^{2,2}(B_\frac{r}{2}(x_0) ) \\ & \; \textrm{s.t.} \; w_{\mid_{\partial B_\frac{r}{2}(x_0) }} \equiv g_1,  \nabla w_{\mid_{\partial B_\frac{r}{2}(x_0) }} \equiv g_2  \Bigg\rbrace, 
\end{align*}  
where '$\equiv$' here means equality in the trace sense. In particular one has 
\begin{equation}\label{eq:4.25}
\int_{B_\frac{r}{2}(x_0) } (\Delta h)^2 \; \mathrm{d}x \leq \int_{B_\frac{r}{2}(x_0) } (\Delta u)^2 \; \mathrm{d}x
\end{equation}
and equality holds if and only if $h \equiv u$, by strict convexity of the energy. Now define 
\begin{equation*}
\widetilde{u}(x) := \begin{cases} u(x) & x \not \in B_\frac{r}{2}(x_0), \\ h(x) & x \in B_\frac{r}{2}(x_0). \end{cases}
\end{equation*}
Since $\widetilde{u}$ has the right regularity and the same boundary data as $u$ one obtains that $\widetilde{u} \in \A(u_0)$. Therefore one can compute with \eqref{eq:4.25}
 \begin{align*}
\E(u) & \leq \E(\widetilde{u}) = \int_{\Omega \setminus B_\frac{r}{2}(x_0) } (\Delta u )^2 \; \mathrm{d}x + \int_{B_\frac{r}{2}(x_0) } (\Delta h)^2 \; \mathrm{d}x \\ & \quad \quad \quad \quad \quad \quad + |\{ u >0  \} \cap \Omega \setminus B_\frac{r}{2}(x_0) | + | \{ h > 0 \} \cap B_\frac{r}{2}(x_0) | 
\\ & \leq  \int_{\Omega \setminus B_\frac{r}{2}(x_0) } (\Delta u )^2 \; \mathrm{d}x + \int_{B_\frac{r}{2}(x_0) } (\Delta u)^2  \; \mathrm{d}x + |\{ u >0  \} \cap \Omega \setminus B_\frac{r}{2}(x_0) | + |B_\frac{r}{2}(x_0) |.
\end{align*}
Now note that $|B_\frac{r}{2}(x_0) | = |B_\frac{r}{2}(x_0) \setminus \{ x_0 \} | = |\{u >0 \} \cap B_\frac{r}{2}(x_0) | $, as we explained in the beginning of the proof. Therefore we obtain
\begin{equation*}
\E(u) \leq \E(\widetilde{u}) \leq \int_\Omega (\Delta u)^2  \; \mathrm{d}x + |\{ u >0 \}| = \E(u). 
\end{equation*}
This means in particular that all estimates used on the way have to hold with equality. Since we used estimate \eqref{eq:4.25}, equality holds in \eqref{eq:4.25} and from this we can infer (see discussion below \eqref{eq:4.25}) that $h = u$. In particular $u \in C^\infty (\overline{B_\frac{r}{2}(x_0)})$. This however is a contradiction to $(\Delta u)^*(x_0) = \infty$ and the claim follows. 
 \end{proof}
\section{Nodal Set and Biharmonic Measure} 
  In this section we are finally able to understand the regularity of the free boundary $\{ u = 0 \}$ and - as a byproduct - the measure $\mu$ of \eqref{eq:bihammeas}. The fact that $\nabla u $ does not vanish on $\{u = 0 \}$ and $u \in C^1(\Omega)$ makes $\{u = 0\}$ already a $C^1$-manifold. By deriving \eqref{eq:bihame} for $u$, we can give a rigorous version of the formal statement \eqref{eq:formmeas}. Afterwards we use this equation to obtain $C^2$ for $u$ and as a result the same additional regularity for $\{ u = 0 \}$. 
 \begin{lemma}[The Measure-Theoretic Boundary]\label{lem:measthebou}
 Let $u \in \A(u_0)$ be a minimizer. Then 
 \begin{equation}\label{eq:438}
 \partial^* \{ u > 0 \} = \{ u = 0 \} ,
 \end{equation}
 \end{lemma}
 \begin{proof}
 For the '$\supset$' inclusion in \eqref{eq:438} note that $x_0 \in \{ u = 0 \}$ implies by Corollary \ref{cor:nosing} that  $\nabla u(x_0) \neq 0 $. Moreover one has
 \begin{align*}
 \frac{|\{u > 0\} \cap B_r(x_0) | }{|B_r(x_0)|} & = \frac{1}{|B_1(0)| r^2}\left\vert  \left\lbrace rx : x \in B_1(0) \; \textrm{s.t} \;  u(x_0 + rx) > 0 \right\rbrace \right\vert 
 \\ & =  \frac{1}{|B_1(0)|}\left\vert  \left\lbrace x \in B_1(0) :   u(x_0 + rx) > 0 \right\rbrace \right\vert
 \\ & =  \frac{1}{|B_1(0)|}\left\vert  \left\lbrace x \in B_1(0) :   \frac{u(x_0 + rx)}{r}  > 0 \right\rbrace \right\vert.
 \end{align*}
Since the expression in the measure term converges uniformly in $r$ to $\nabla u(x_0) \cdot x $ we get by Fatou's lemma   
 \begin{align*}
 \overline{\theta}(\{u > 0 \} ,x_0) & = \limsup_{r\rightarrow 0 }  \frac{|\{u > 0\} \cap B_r(x_0) | }{|B_r(x_0)|} \\ &  \geq \frac{1}{|B_1(0)|}\left\vert  \left\lbrace x \in B_1(0) :  \nabla u(x_0) \cdot x  > 0 \right\rbrace \right\vert = \frac{1}{2},
 \end{align*}
 as $\{ x : \nabla u(x_0) \cdot x > 0 \}$ defines a half plane through the origin. 
 Similarly one shows $\overline{\theta}(\{ u \leq 0 \},x_0)  \geq \frac{1}{2} > 0 $ and hence the inclusion is shown. 
For the remaining inclusion take $x_0 \in \partial^* \{ u > 0 \} $. If $u(x_0) > 0$ then there exists $r_0 > 0 $ such that $ u > 0$ on $B_{r_0} (x_0)$ and this implies by definition of $\overline{\theta}$ that $\overline{\theta} (\{u \leq  0 \} ,x_0 ) = 0$.  Similarly one shows that $u(x_0)< 0 $ implies that $\theta ( \{ u > 0 \} , x_0 ) = 0 $. Hence $u(x_0) = 0 $ and the claim follows.
 \end{proof}
 We will now characterize the measure found in \eqref{eq:bihammeas} using an inner variation technique that has led to rich insights in \cite{Valdinoci}.
 \begin{lemma}[Noether Equation] \label{lem:noether}
 Let $u \in \A(u_0)$ be a minimizer. Then 
 \begin{equation}\label{eq:finper}
 \int_{\Omega_{\epsilon_0}^C}  \chi_{\{u> 0 \}} \mathrm{div}(\phi)\dx  = -  \int_{\Omega_{\epsilon_0}^C} \nabla u \cdot \phi \; \mathrm{d}\mu  \quad \forall \phi \in C_0^\infty(\Omega_{\epsilon_0}^C, \mathbb{R}^2),
 \end{equation}
 where $\mu$ is the biharmonic measure from \eqref{eq:bihammeas}.
 \end{lemma}
 \begin{proof}
To compactify notation, we will leave out the `$\cdot$' to indicate the dot product for this proof.
 From \cite[Lemma 4.3]{Valdinoci} follows that for each $\phi \in C_0^\infty(\Omega; \mathbb{R}^2)$ one has
  \begin{equation}\label{eq:515}
2  \int_\Omega \Delta u  \sum_{m = 1}^{2} \left(  2 \nabla (\partial_m u) \cdot \nabla \phi^m  + \partial_m u\Delta \phi^m \right) \dx - \int_\Omega ( ( \Delta u)^2 + \chi_{\{ u > 0 \} } ) \mathrm{div}(\phi) \dx = 0  .
 \end{equation}
 Fix  $\phi \in C_0^\infty(\Omega_{\epsilon_0}^C; \mathbb{R}^2)$. Then there is $\beta \in (0,1)$ such that $\nabla u\cdot \phi \in W_0^{2,2-\beta}(\Omega_{\epsilon_0}),$ by Corollary \ref{cor:minregu}. Observe that $\nabla u \cdot \phi$ is a valid test function for \eqref{eq:bihammeas} (cf. Lemma \ref{lem:bihmmeasmore}).   Starting from \eqref{eq:515} Corollary \ref{cor:minregu} we can  use \eqref{eq:bihammeas} to find
 \begin{align*}
 \int_{\Omega_{\epsilon_0}^C}  \chi_{ \{u > 0 \}}&  \mathrm{div}(\phi) \dx   =2  \int_{\Omega_{\epsilon_0}^C} \Delta u  \sum_{m = 1}^{2} \left(  2 \nabla (\partial_m u) \nabla \phi^m  + \partial_m u\Delta \phi^m \right) \dx \\ & \quad \quad \quad  \quad \quad \quad \quad \quad \quad \quad \quad \quad \quad \quad \quad \quad \quad \quad \quad \quad \quad - \int_{\Omega_{\epsilon_0}^C}  ( \Delta u)^2  \mathrm{div}(\phi) \dx
 \\ & =  2  \int_{\Omega_{\epsilon_0}^C} \Delta u  \sum_{m = 1}^{2} \left( 
 \Delta (\partial_m u \phi^m - \phi^m  \Delta (\partial_m u) \right) \dx -  \int_{\Omega_{\epsilon_0}^C}   (\Delta u )^2 \mathrm{div}(\phi) \dx  
 \\ & =2  \int_{\Omega_{\epsilon_0}^C} \Delta u \Delta ( \nabla u \phi) \dx - \int_{\Omega_{\epsilon_0}^C} ( \phi \nabla (\Delta u)^2 + (\Delta u)^2 \mathrm{div}(\phi) \dx
 \\ & = - \int_{\Omega_{\epsilon_0}^C} \nabla u \phi \; \mathrm{d}\mu - \int_{\Omega_{\epsilon_0}^C} \mathrm{div} ( (\Delta u)^2 \phi ) \dx  . 
 \end{align*}
The second integral vanishes by the Gauss divergence theorem and the claim follows. 
 \end{proof}
\begin{cor}\label{cor:finhau}
Let $u \in \A(u_0) $ be a minimizer. Then $\{ u > 0 \} $ has finite perimeter in $\Omega$ and $\mathcal{H}^1(\{ u = 0  \} ) < \infty.$
\end{cor} 
\begin{proof}
We first show that $\{u >0 \}$ has finite perimeter in $\Omega_{\epsilon_0}^C$. Observe that by  \eqref{eq:finper} one has for each $\phi \in C_0^\infty ( \Omega_{\epsilon_0}^C ; \mathbb{R}^2) $ such that $\sup_{\Omega_{\epsilon_0}^C}  |\phi| \leq 1$,
\begin{align*}
\int_{\Omega_{\epsilon_0}^C } \chi_{\{u > 0\} } \mathrm{div} ( \phi) \dx =  -\int_{\Omega_{\epsilon_0}^C} \nabla u \cdot \phi \; \mathrm{d}\mu  \leq   \mu(\Omega)\sup_{x \in \Omega_{\epsilon_0}^C }|\nabla u (x)| .
\end{align*} 
The quantity on the right hand side is finite since by Corollary \ref{cor:minregu} there is $\beta \in (0,1)$ such that  $\nabla u \in W^{2,2-\beta}(\Omega_{\epsilon_0}^C) \subset C(\overline{\Omega_{\epsilon_0}^C})$. By \cite[Thm.1(i), Sect.5.9]{EvGar} we conclude that $\mathcal{H}^1(\partial^* \{ u > 0 \} \cap \Omega_{\epsilon_0}^C) < \infty$.
By Lemma \ref{lem:measthebou} we have $\partial^*\{ u > 0 \} =  \{u = 0 \} \subset \Omega_{\epsilon_0}^C$. Therefore, $\mathcal{H}^1( \Omega \cap  \partial^*\{ u >0 \} ) < \infty$ and by \cite[Thm.1, Sect.5.11]{EvGar} we obtain that $\{u > 0 \} $ has finite perimeter in $\Omega$. By Lemma \ref{lem:measthebou} we conclude
$\infty > \mathcal{H}^1( \partial^* \{ u > 0 \} ) = \mathcal{H}^1 ( \{ u = 0 \} )$. \qedhere

\end{proof}
\begin{lemma}[Biharmonic Measure and Hausdorff Measure] \label{lem:bihammeasreg}
Let $ A \subset \Omega$ be a Borel set and $u \in \A(u_0)$ be a minimizer. Then 
\begin{equation}\label{eq:542}
\mu ( A  ) = \int_A \frac{1}{|\nabla u |} \; \mathrm{d}\mathcal{H}^1  \llcorner_{ \{ u = 0  \} }.
\end{equation}
\end{lemma}
\begin{proof}
 We first prove the formula
 \begin{equation}\label{eq:543}
 \int_{\Omega_{\epsilon_0}^C} \phi |\nabla u |^2 d\mu = \int_{\{ u = 0\} } \phi |\nabla u| \; \mathrm{d}\mathcal{H}^1 \quad \forall \phi \in C_0(\Omega_{\epsilon_0}^C). 
\end{equation}
By density, it suffices to prove the claim for $\phi \in    C_0^\infty(\Omega_{\epsilon_0}^C). $
For $\epsilon >0 $, let $\rho_\epsilon$ be the standard mollifier and define $f_\epsilon := (\phi \nabla u )* \rho_\epsilon$. Now note that $f_\epsilon$ lies in $C_0^\infty(\Omega_{\epsilon_0}^C)$ for appropriately small $\epsilon > 0$ and $f_\epsilon$ converges uniformly to $ \phi \nabla u$. By \eqref{eq:finper} and the fact that $\{ u >0 \}$ has finite perimeter in $\Omega_{\epsilon_0}^C$ by Corollary \ref{cor:finhau} we obtain with \cite[Thm.1,Sect.5.9]{EvGar} that
\begin{align}
\int_{\Omega_{\epsilon_0}^C} \phi |\nabla u|^2 \; \mathrm{d}\mu & =  \lim_{\epsilon \rightarrow 0 } \int_{\Omega_{\epsilon_0}^C} \nabla u \cdot f_\epsilon \; \mathrm{d}\mu \nonumber \\
& \label{eq:545} = -  \lim_{\epsilon \rightarrow 0 }  \int_{\Omega_{\epsilon_0}^C} 
 \chi_{\{ u >0 \}} \mathrm{div} (f_\epsilon) \dx 
 =-  \lim_{\epsilon \rightarrow 0 } \int_{\partial^* \{ u > 0 \}} f_\epsilon \cdot \nu_{ \{ u > 0 \} }\; \mathrm{d}\mathcal{H}^1
 \\ & = -  \lim_{\epsilon \rightarrow 0 } \int_{ \{ u = 0 \}} f_\epsilon \cdot \nu_{ \{ u > 0 \} }\; \mathrm{d}\mathcal{H}^1, \nonumber
\end{align}  
where $\nu_{\{ u > 0\}}$ denotes the measure theoretic unit outer normal to $\{ u >0 \}$, cf. \cite[Thm.1,Sect.5.9]{EvGar}. Since by Remark \ref{rem:manifold}, $\{ u = 0 \}$ is locally a $C^1$-regular level set one obtains immediately that $\nu_{\{ u > 0\}}(x) = \frac{-\nabla u(x)}{|\nabla u (x) |}$. Together with the fact that $\mathcal{H}^1(\{u = 0 \})< \infty$ by Corollary \ref{cor:finhau} we obtain \eqref{eq:543}. Since $\nabla u$ is a continuous function that does not vanish on $\{u = 0 \} \subset \Omega_{\epsilon_0}^C$ there also exists some $\epsilon >0 $ such that $\overline{ B_\epsilon(\{ u = 0 \} )} \subset \Omega_{\epsilon_0}^C$ and $\nabla u $ does not vanish on $B_\epsilon( \{ u = 0 \} )$. Fix  $\eta \in C_0^\infty( B_\epsilon ( \{ u = 0 \}))$ arbitrarily such that $\eta \equiv 1$ on $\{u = 0 \}$. Now suppose that $\psi \in C_0(\Omega)$. Note that $\eta \equiv 1$ on $\mathrm{supp}(\mu)$ and $\frac{\psi}{|\nabla u|^2} \eta \in C_0(\Omega_{\epsilon_0}^C)$. Therefore one has by \eqref{eq:543}
\begin{align*}
\int \psi \; \mathrm{d}\mu = \int_{\Omega_{\epsilon_0}^C} \frac{\psi \eta}{|\nabla u|^2} |\nabla u |^2 \; \mathrm{d}\mu = \int_{\{u = 0\}} \frac{\psi \eta}{|\nabla u|^2} |\nabla u| \; \mathrm{d}\mathcal{H}^1 = \int_{ \{u = 0  \} } \frac{\psi}{|\nabla u|} \; \mathrm{d}\mathcal{H}^1.
\end{align*}
From there the claim is easy to deduce by standard arguments in measure theory. 
\end{proof}
Having now characterized the measure $\mu$ explicitly, one can obtain classical regularity with the representation \eqref{eq:347}. The details will be discussed in Appendix \ref{sec:regunod}. 
\begin{lemma} [$C^2$-Regularity, Proof in Appendix \ref{sec:regunod}] \label{lem:regunod}
Let $u \in \A(u_0)$ be a minimizer. Then $ u \in C^2( \Omega)$.
\end{lemma}

\section{Proof of Theorem \ref{thm:1.1}}

\begin{proof}[Proof of Theorem \ref{thm:1.1}] 
We first recall parts of the statement that have already been proved on the way: The $C^2$-regularity of $u$ and the property that $\nabla u \neq 0 $ whenever $u =0 $ follow from Remark \ref{rem:manifold} and Lemma \ref{lem:regunod}. The $W^{3,2-\beta}_{loc} $-regularity follows from Lemma \ref{lem:biham} and Corollary \ref{cor:minregu}.
By Corollary \ref{lem:corsubham} we can infer that $\Delta u \geq 0 $. We show now that $\{ u = 0 \}$ is a closed connected $C^2$-hypersurface. First $\{ u = 0 \}$ is a $C^2$-manifold as zero level set of a $C^2$-function with nonvanishing gradient on $\{u = 0\}$. Note that $\{ u = 0 \}$ is orientable as $\nu = \frac{\nabla u}{|\nabla u |}$ defines a continuous normal vector field. Furthermore, each connected component of $\{ u = 0 \}$ is a connected, orientable $C^2$-manifold. Note that $\{ u= 0 \}$ has only finitely many connected components $(S_i)_{i = 1}^N$ since it is compact.
 We also claim that each connected component of $\{u = 0\}$ is compact. Indeed, connected components of topological spaces are closed the the same space, cf. \cite[Exercise 1.6.1]{Lawson}, and closed subsets of compact sets are compact. All in all, each connected component of $\{u = 0 \}$ is a compact, orientable, connected $C^2$-manifold. By the Jordan-Brower seperation theorem (see \cite{Lima}), we infer that for each $i \in \{ 1, ..., N \}$ the set $\mathbb{R}^2 \setminus S_i$ has two disjoint connected components, say $G_i$ and  $\mathbb{R}^2 \setminus ( G_i \cup S_i) $ the boundary of both of which is $S_i$. We claim that one of these two components is a subset of $\Omega$. For if not, one can find  an $x_1 \in G_i \setminus \Omega$ as well as $x_2 \in (\mathbb{R}^2 \setminus ( G_i \cup S_i)) \setminus \Omega$. One can then connect $x_1$ and $x_2$ with a continuous path lying in $\mathbb{R}^2 \setminus \Omega \subset \mathbb{R}^2 \setminus S_i $. This is a contradiction since $G_i$ and $\mathbb{R}^2 \setminus (G_i \cup S_i)$ are two different path components of $\mathbb{R}^2 \setminus S_i$.  Without loss of generaliy $G_i$ is contained in $\Omega$.  
Note that $G_i$ has positive distance of $\partial \Omega$ since $\overline{G_i}$ is compact $\inf_{G_i} \dist(\cdot, \partial \Omega)$ is attained in $\partial G_i = S_i$.

Since $u$ is subharmonic in $\Omega$ by Corollary \ref{lem:corsubham} we get that either $u \equiv 0$ in $G_i$ or $ u< 0 $ in $G_i$ by the strong maximum principle for subharmonic functions. The first possibility is excluded since $\nabla u $ does not vanish on $\{ u = 0\}$ as we already showed. 

We show now that $G_i\cap G_j = \emptyset$  for all $i \neq j$. Since $u< 0 $ in $G_i$ for all $i$, we get $S_j \cap G_i = \emptyset$ for all $j \neq i$. Therefore 
\begin{equation*}
\partial (G_i \cap G_j) \subset (S_j \cap \overline{G_i}) \cup (S_i \cap \overline{G_j}) = (S_j \cap G_i) \cup (S_i \cap G_j) \cup (S_i \cap S_j)  = \emptyset
\end{equation*}
 for all $i \neq j$ . This means that $\mathbb{R}^2$ is the disjoint union of $G_i \cap G_j$ and the interior of $(G_i \cap G_j)^c$. Since $\mathbb{R}^2$ is connected we obtain that $G_i \cap G_j = \emptyset$ for $i \neq j$. We show next that $\{ u< 0 \} = \bigcup_{i = 1}^N G_i$. Suppose that there is a point $\widetilde{x} \in \Omega \setminus \bigcup_{i = 1}^N G_i$ such that $u( \widetilde{x} ) < 0 $. Let $\widetilde{r}:= \sup\{ r > 0 : B_r( \widetilde{x} ) \subset \{ u < 0 \} \}$. Observe that $\widetilde{r}> 0 $ because of continuity of $u$. Note that $\overline{B_{\widetilde{r}}(\widetilde{x})} \subset \Omega$ because $ u > 0 $ on $\partial \Omega$.  
Hence, $\overline{B_{\widetilde{r}}( \widetilde{x} )} $ touches some $S_j$ tangentially. Note also that $\{u < 0 \} $ in $B_{\widetilde{r}} ( \widetilde{x} )$ and $B_{\widetilde{r}} ( \widetilde{x} ) \cap G_j = \emptyset$ since $\widetilde{x} \in \mathbb{R}^2 \setminus( G_j \cup S_j)$ and $B_{\widetilde{r}} ( \widetilde{x} ) $ can only intersect one connected component of $\mathbb{R}^2 \setminus S_j$. Let $p \in S_j$ be a point where $\overline{B_{\widetilde{r}}( \widetilde{x} )} $ touches $S_j$. Now observe that $t \mapsto u(p + t \nabla u(p))$ is continuously differentiable in a neighborhood of $p$ as $u \in C^1(\Omega)$. Therefore 
 \begin{equation*}
 \frac{d}{dt}_{\mid_{t = 0 }} u(p +t \nabla u(p) ) = | \nabla u (p)|^2  > 0
 \end{equation*}
 and hence there is $t_0 > 0 $ such that $\frac{d}{dt} u(p +t \nabla u(p) ) \geq \frac{1}{2} | \nabla u(p)|^2$ for each $t \in (-t_0,t_0)$. In particular the fundamental theorem of calculus yields that 
\begin{equation} \label{eq:positi}
 u(p + t \nabla u(p) ) > 0   \quad \forall t \in (0,t_0).
\end{equation}  
 Since $B_{\widetilde{r}}(\widetilde{x})$ touches $S_j$ tangentially at $p$ the exterior normal of $B_{\widetilde{r}}(\widetilde{x})$ at $p$ is given by $\nu = \pm \frac{\nabla u}{|\nabla u|}$. In case that $\nu = + \frac{\nabla u}{|\nabla u|}$ the exterior unit normal coincides with the exterior unit normal of $G_j$. Since $G_j$ and $B_{\widetilde{r}}(\widetilde{x})$ both satisfy the interior ball condition (see \cite[Remark 4.3.8]{Han}), we can now force a small ball into $G_j \cap B_{\widetilde{r}}(\widetilde{x})$ which is a contradiction to the fact that $G_j \cap  B_{\widetilde{r}}(\widetilde{x}) = \emptyset$. Therefore  $\nu = - \frac{\nabla u}{|\nabla u|}$ and hence there is $t_1 > 0 $ such that $ p + t \nabla u (p)$ lies in $B_{\widetilde{r}} ( \widetilde{x})$ for each $t \in (0,t_1)$. Choosing $t := \frac{1}{2} \min\{ t_0, t_1 \}$ we obtain a contradiction since $p  +t \nabla u(p) \in B_{\widetilde{r}}(\widetilde{x})$  and $ u( p + t \nabla u(p) ) > 0 $ according to  \eqref{eq:positi}, which is a contradiction to the choice of $B_{\widetilde{r}}(\widetilde{x})$. We have shown \eqref{eq:geb}.
 Given this, we get the following chain of set inclusions:  
\begin{equation*}
\{u = 0 \} = \bigcup_{i = 1}^\infty S_i = \bigcup_{i = 1}^\infty \partial G_i  \subset \partial\{ u < 0 \} \subset \{ u= 0 \},
\end{equation*}
where we used the continuity of $u$ in the last step. We obtain that $\partial \{ u< 0 \}  = \{ u = 0 \}$, which was also part of the statement.  The property that $\{ u = 0\}$ has finite $1$-Hausdorff measure follows from Corollary \ref{cor:finhau}. The only statement that remains to show is \eqref{eq:bihame}. We first show \eqref{eq:bihame} for $\phi \in C_0^\infty(\Omega)$. By \eqref{eq:bihammeas} one has
\begin{equation*}
2\int_\Omega \Delta u \Delta \phi = -\int_\Omega \phi \; \mathrm{d}\mu \quad \forall \phi \in C_0^\infty(\Omega)
\end{equation*} 
for a measure $\mu$ with $\mathrm{supp}(\mu)= \{ u= 0 \}$ which was examined more closely in Lemma \ref{lem:bihammeasreg}. From this lemma we can conclude that 
\begin{align*}
\mu(A) & = \int_A \frac{1}{|\nabla u | } \; \mathrm{d}\mathcal{H}^1 \llcorner_{ \{u = 0 , \nabla u \neq 0 \} } = \int_{\{u = 0 \}} \chi_A \frac{1}{|\nabla u | } \; \mathrm{d}\mathcal{H}^1. 
\end{align*} 
Using this representation of $\mu$ we obtain \eqref{eq:bihame} for $\phi \in C_0^\infty(\Omega)$ and by density also for $ \phi \in W_0^{2,2}(\Omega)$. Now suppose that $\phi \in W^{2,2}(\Omega) \cap W_0^{1,2}(\Omega)$. Choose $\eta \in C_0^\infty(\Omega_{\epsilon_0}^C)$ such that $0 \leq \eta \leq 1$ and $\eta \equiv 1$ in a neighborhood of $\{u \leq 0 \}$ that is compactly contained in $\Omega_{\epsilon_0}^C$ and rewrite $\phi = \phi \eta + \phi ( 1- \eta)$. Observe that $\phi (1- \eta)$ lies in $W^{2,2}(\Omega) \cap W_0^{1,2}(\Omega)$ and is compactly supported in $\{u > 0 \}$. By Lemma \ref{lem:biham} we infer that 
\begin{equation}\label{eq:wegmitrand}
2\int_\Omega \Delta u \Delta (\phi(1- \eta) ) \dx  = 0 . 
\end{equation}  
Note that $\phi \eta \in W_0^{2,2}(\Omega)$ as $\eta$ is compactly supported in $\Omega$. Using \eqref{eq:wegmitrand} and that we have already shown \eqref{eq:bihame} for $W_0^{2,2}$-test functions we find
\begin{equation*}
2 \int_\Omega \Delta u \Delta \phi \dx = 2\int_\Omega \Delta u \Delta (\eta \phi) \dx= - \int_{\{ u= 0 \} } \phi \eta  \frac{1 }{|\nabla u |} \; \mathrm{d}\mathcal{H}^1 .
\end{equation*}
Since $\eta \equiv 1$ on a neighborhood of $\{ u= 0 \}$ we obtain the claim. 
\end{proof}
\begin{cor}
Let $u \in \A(u_0)$ be a minimizer. Then $\partial \{ u > 0 \} = \{ u = 0 \} \cup \partial \Omega$. In particular $ \{u > 0 \}$ has $C^2$-boundary.  
\end{cor}
\begin{proof}
Recalling \eqref{eq:438} we find that 
\begin{equation*}
\partial \{u > 0\} = (\partial \{ u > 0 \} \cap \Omega) \cup \partial \Omega \supset (\partial^* \{ u > 0 \} \cap \Omega) \cup \partial \Omega  \supset \{ u =0\} \cup \partial \Omega .
\end{equation*}
The other inclusion $\partial \{ u >0 \} \subset \partial \Omega \cup \{ u = 0 \}$ is immediate by continuity of $u$. The rest of the claim follows from Theorem \ref{thm:1.1}. 
\end{proof}
  \section{Measure of the Negativity Region}
We have aleady discovered in \eqref{eq:infbound}, Example \ref{ex:iotapos}, and Remark \ref{rem:nontriv} that for `small' boudary values $u_0$ the energy of minimizers falls below $|\Omega|$ and the nodal set is nontrivial. On the contrary, for `large' boundary values $u_0$, one gets minimizers with trivial nodal set, see Remark \ref{rem:iotafin}.  In this section, we want to derive some estimates that ensure one of the two cases.

\begin{lemma}[Universal Bound for Biharmonic Measure] 
Let $u \in \A(u_0)$ be a minimizer. Then 
\begin{equation}\label{eq:89}
\int_{\{ u = 0 \} } \frac{1}{|\nabla u | } \; \mathrm{d}\mathcal{H}^1 \leq \frac{2|\{u < 0 \}|}{\inf_{\partial \Omega } u_0 } \leq \frac{2|\Omega|}{\inf_{\partial \Omega} (u_0)}
\end{equation}
\end{lemma}
\begin{proof}
Let $h \in W^{2,2}(\Omega)$ be a solution of 
\begin{equation*}
\begin{cases}
\Delta h = 0 & \mathrm{in} \; \Omega, \\
h = u_0  & \mathrm{on} \; \partial \Omega .
\end{cases}
\end{equation*}
Note that by elliptic regularity and the trace theorem, see \cite[Theorem 8.12]{Gilbarg}, $h$ lies actually in $W^{2,2}(\Omega)$ and $h - u_0 \in W_0^{1,2}(\Omega)$. Observe that $u - h = (u- u_0) + (u_0 - h) \in W^{2,2}(\Omega) \cap W_0^{1,2}(\Omega)$. We obtain with \eqref{eq:bihame}  
\begin{equation}\label{eq:hasuii}
\int_\Omega \Delta u \Delta (u-h) \dx=  - \frac{1}{2}\int_{\{ u= 0 \}}  \frac{u - h}{|\nabla u | } \; \mathrm{d}\mathcal{H}^1 = \frac{1}{2}\int_{\{ u= 0 \}}  \frac{ h}{|\nabla u | }  \; \mathrm{d}\mathcal{H}^1 . 
\end{equation}
For the left hand side we can estimate using harmonicity of $h$ and  \eqref{eq:infbound} 
\begin{equation*}
\int_\Omega \Delta u \Delta (u-h) \dx = \int_\Omega (\Delta u)^2  \dx = \E(u) - |\{ u >0 \}| \leq  |\Omega| - |\{ u >0 \}| = |\{u < 0 \}| ,
\end{equation*}
where we used that $|\{u = 0 \}| = 0 $ by Theorem \ref{thm:1.1}. 
Therefore \eqref{eq:hasuii} implies that
\begin{equation*}
|\{u < 0 \}| \geq \frac{1}{2} \int_{\{ u= 0 \}}  \frac{h}{|\nabla u | } \; \mathrm{d}\mathcal{H}^1.
\end{equation*}
By the maximum priciple and the construction of $h$ we obtain that $h \geq \inf_{\partial \Omega} u_0$ and hence 
\begin{equation*}
|\{ u< 0 \}| \geq \frac{1}{2} \inf_{\partial\Omega} ( u_0 )  \int_{\{ u= 0 \}}  \frac{1}{|\nabla u | } \; \mathrm{d}\mathcal{H}^1. \qedhere
\end{equation*}
\end{proof}
 
\begin{cor}[One-phase Solutions for Large Boundary Values] \label{cor:onephase}
Let $\Omega$, $u_0$ be as in Definition \ref{def:adm} and $u \in \A(u_0)$ be a minimizer. Then, either $ \{ u = 0\}= \emptyset $ or 
\begin{equation*}
| \{ u < 0 \} | \geq 2\pi \inf_{\partial \Omega } u_0 
\end{equation*} 
\end{cor}
\begin{proof}
Using that $| \{ u = 0 \}| =0$ by Theorem \ref{thm:1.1} and \eqref{eq:infbound} as well as  the Cauchy-Schwarz inequality and Gauss divergence theorem  we get
\begin{align*}
|\{ u < 0 \} |  & = |\Omega | - |\{ u > 0 \} | \geq \E(u) - |\{ u > 0 \} | = \int_\Omega (\Delta u )^2 \dx 
\\ & \geq \int_{\{ u<0 \} } (\Delta u )^2  \dx = \frac{1}{| \{ u < 0 \} |} \left( \int_{ \{ u < 0 \} } (\Delta u ) \dx \right)^2 
\\ & = \frac{1}{|\{ u < 0 \}|}  \left( \int_{\partial \{ u <0 \}} \nabla u \cdot \nu  \; \mathrm{d}\mathcal{H}^1 \right)^2 .
\end{align*}
Note that the exterior outer normal of $\{ u< 0 \}$ is given by $\nu = \frac{\nabla u }{|\nabla u |}$ and therefore we obtain with Theorem \ref{thm:1.1}
\begin{equation}\label{eq:711}
|\{ u < 0 \} |^2 \geq \left( \int_{\partial \{ u < 0 \}  } |\nabla u|  \; \mathrm{d}\mathcal{H}^1 \right)^2   = \left( \int_{\{ u = 0 \} } |\nabla u|  \; \mathrm{d}\mathcal{H}^1 \right)^2 .
\end{equation}
Now observe that by the Cauchy Schwarz inequality and \eqref{eq:89} we get  
\begin{align*}
\mathcal{H}^1( \partial \{ u >  0 \}  ) & \leq \left( \int_{ \{u= 0 \} } \frac{1}{| \nabla u |} \; \mathrm{d} \mathcal{H}^1 \right)^\frac{1}{2} \left( \int_{ \{u= 0 \} } {| \nabla u |} \; \mathrm{d} \mathcal{H}^1 \right)^\frac{1}{2}  \\ & \leq \left(2 \frac{|\{ u < 0 \}| }{\inf_{\partial \Omega} u_0 }  \right)^\frac{1}{2}   \left( \int_{ \{u= 0 \} } {| \nabla u |} \; \mathrm{d} \mathcal{H}^1 \right)^\frac{1}{2}.
\end{align*}
Rearranging and plugging into \eqref{eq:711} we find 
\begin{equation*}
|\{ u < 0 \} |^2 \geq \frac{\mathcal{H}^1(\partial \{ u <0 \} )^4}{4|\{ u <0 \}|^2} \left( \inf_{\partial \Omega} u_0 \right)^2 .
\end{equation*}
Using the isoperimetric inequality, see \cite[Theorem 14.1]{Maggi} we get that 
\begin{equation}\label{eq:negativityset}
|\{ u < 0 \} |^2 \geq \frac{1}{4} \frac{\mathcal{H}^1(\partial B_1(0) )^4}{|B_1(0)|^2}  \left(\inf_{\partial \Omega} u_0  \right)^2= 4\pi^2 \left( \inf_{\partial \Omega } u_0 \right)^2. \qedhere
\end{equation}
\end{proof}
\begin{remark}
This proves in particular that $\inf_{\partial \Omega} u_0 > \frac{|\Omega|}{2\pi}$ implies $\{u = 0 \}= \emptyset$, which is a lot better than the bound in Remark \ref{rem:iotafin}, at least for domains $\Omega$ with big Lebesgue measure.   
\end{remark}

\section{A Non-Uniqueness Result}\label{sec:nonuni}

 \begin{definition}[The Candidate for Non-Uniqueness] \label{def:cand}
 Let $\Omega = B_1(0)$. For $C > 0  $ let $\A(C)$ denote the admissible set associated to the boundary function $u_0 \equiv C$, see Definition \ref{def:adm}. 
 \begin{equation*}
 \iota := \sup \{ C > 0  : \inf_{u \in \A(C)} \E(u) < |B_1(0)| \} 
 \end{equation*}
 \end{definition}
 \begin{remark}
 Note that $\frac{1}{8\sqrt{2}} \leq \iota < \infty$ by Example \ref{ex:iotapos} and Remark \ref{rem:iotafin}. 
 \end{remark}
 \begin{lemma}[Energy in the Limit Case] \label{lem:limica}
 Let $\iota$ be as in Definition \ref{def:cand}. Then 
 \begin{equation}\label{eq:intaen}
 \inf_{ u \in \A(\iota) } \E(u) = |B_1(0)|. 
 \end{equation}
 \end{lemma}
 \begin{proof}
 One inequality is immediate by \eqref{eq:infbound}. Now suppose that 
 \begin{equation}\label{eq:73}
 \inf_{ u \in \A(\iota) } \E(u) < |B_1(0)|. 
 \end{equation}
 Then by Theorem \ref{thm:1.1} and Remark \ref{lem:eximin} there exists a minimizer $u_\iota$ such that $ \{u_{\iota} =0 \} $ has finite $1-$Hausdorff measure and $\E(u_\iota) < |B_1(0)|$. Note that for each $\epsilon > 0$, $u_\iota + \epsilon \chi_{B_1(0)}$ is admissible for $u_0 \equiv \iota + \epsilon$. By the choice of $\iota$ we get
 \begin{equation*}
 |B_1(0)| \leq \E(u_{\iota} + \epsilon  ) = \int_{B_1(0)} (\Delta u_{\iota})^2 \dx + |\{ u_\iota + \epsilon > 0 \}| \quad \forall \epsilon > 0. 
 \end{equation*}
 Hence 
 \begin{equation*}
 \int_{B_1(0)} (\Delta u_{\iota})^2 \dx  +| \{ u_\iota > -  \epsilon   \}| \geq |B_1(0)| 
 \end{equation*}
 Letting $\epsilon >0 $ monotonically from above, we obtain with \cite[Theorem 1 in Section 1.1]{EvGar} that 
 \begin{equation}\label{eq:contiii}
  \int_{B_1(0)} (\Delta u_{\iota})^2 \dx +| \{ u_\iota \geq 0    \}| \geq |B_1(0)|. 
 \end{equation}
 As we already pointed out, $\{ u_\iota = 0 \}$ is a set of finite $1$-Hausdorff measure and hence a Lebesgue null set, see \cite[Section 2.1, Lemma 2]{EvGar}. Therefore \eqref{eq:contiii} can be reformulated to 
 \begin{equation*}
  \int_{B_1(0)} (\Delta u_{\iota})^2 \dx +| \{ u_\iota > 0    \}| \geq |B_1(0)|,
 \end{equation*}
 but the left hand side coincides with $\E(u_\iota)$, which is a contradiction to \eqref{eq:73}. 
 \end{proof}
 
\begin{remark}  \label{rem:wehamini}
Equation \eqref{eq:intaen} already yields one immediate minimizer, namely $u \equiv \iota$. We have to show that there exists yet another minimizer. 
\end{remark}
 
\begin{proof}[Proof of Theorem \ref{thm:nonun}] Let $\iota$ be as in Definition \ref{def:cand}, $\Omega= B_1(0)$ and $u_0 \equiv \iota$. 
Let $(\iota_n)_{n \in \mathbb{N}}$ be a sequence such that $\iota_n \leq \iota_{n+1}< \iota$ for each $n$, $\inf_{w \in \A(\iota_n)} \E(w) < |\Omega|$ and  $\iota_n \rightarrow \iota$ as $n \rightarrow \infty$. Such a sequence exists by the choice of $\iota$, see Definition \ref{def:cand}. For each $n \in \mathbb{N}$ let $u_n \in \A(\iota_n )$ be a  minimizer with boundary values $\iota_n$. By Remark \ref{rem:nontriv} we obtain that 
\begin{equation}\label{eq:neaga}
\inf_{x \in \Omega} u_n(x) \leq 0 .
\end{equation}
We claim that $||u_n||_{W^{2,2}}$ is bounded. Indeed, by \cite[Theorem 2.31]{Sweers} we get for some $C> 0$ independent of $n$
\begin{align*}
||u_n||_{W^{2,2}} &\leq ||\iota_n||_{W^{2,2}}+ C || u_n - \iota_n ||_{W^{2,2}} 
\leq  ||\iota_n||_{L^2} + C || \Delta (u_n - \iota_n )||_{L^2}
\\ & \leq \iota|\Omega|^\frac{1}{2} + C || \Delta u_n||_{L^2} \leq (\iota + C) \sqrt{|\Omega|},
\end{align*}
where we used \eqref{eq:infbound} in the last step. 
Therefore $(u_n)_{n = 1}^\infty$ has a weakly convergent subsequence in $W^{2,2}(\Omega)$, which we call $u_n$ again without relabeling. Let $u \in W^{2,2}(\Omega)$ be its weak limit. Since $u_n - \iota_n \in W_0^{1,2}(\Omega)$ and $W_0^{1,2}(\Omega)$ is weakly closed, we find that $u \in \A( \iota)$. Since $W^{2,2}(\Omega)$ embeds compactly into $C(\overline{\Omega})$, $u_n $ converges also uniformly to $u$. Using \eqref{eq:neaga} we obtain that 
\begin{equation*}
 \inf_{x \in \Omega} u(x) \leq 0 . 
\end{equation*}
In particular, $u$ differs from the function identical to $\iota$ which was already identified in  Remark \ref{rem:wehamini} as a minimizer in $\A(\iota)$. We show now that $u$ is another minimizer in $\A(\iota)$. By Lemma \ref{lem:limica}, the weak lower semicontinuity of the $L^2$ norm with respect to $L^2$-convergence and Fatou's Lemma we get
\begin{align}
|\Omega| & \leq \E(u) = \int_\Omega (\Delta u )^2 \dx + |\{u > 0 \}|\leq
\liminf_{n \rightarrow \infty} \int_\Omega (\Delta u_n)^2 \dx +  \int_\Omega \chi_{\{ u > 0 \} } \dx  \nonumber
\\ & =     \liminf_{n \rightarrow \infty} \int_\Omega (\Delta u_n)^2 \dx +  \int_\Omega \liminf_{n \rightarrow \infty } \chi_{\{ u_n > 0 \} } \dx \nonumber
\\ & \leq  \liminf_{n \rightarrow \infty} \int_\Omega (\Delta u_n)^2 \dx +  \liminf_{n \rightarrow \infty } \int_\Omega  \chi_{\{ u_n > 0 \} } \dx \nonumber
\\ &  \leq \liminf_{n \rightarrow \infty} \left( \int_\Omega (\Delta u_n)^2 \dx + |\{u_n > 0 \}| \right)  = \liminf_{n \rightarrow \infty } \E(u_n) \leq |\Omega| \label{eq:linenerg}.
\end{align}
Therefore $\E(u) = |\Omega| = \inf_{ w \in \A(\iota) }\E(w)$ by \eqref{eq:intaen}, which proves the claim.
\end{proof} 
\section{On Navier Boundary Conditions}\label{sec:fwts}
As we have only shown interior regularity of minimizers in Theorem \ref{thm:1.1} we cannot conclude anything about the behavior of the Laplacian at the boundary. However, the weak formulation of Navier boundary conditions in Definition \ref{def:adm} is equivalent to the strong formulation only provided that $u$ is regular enough to have trace at $\partial \Omega$, see the discussion in \cite[Section 2.7]{Sweers} for details. Provided that the domain $\Omega$ has actually smooth boundary (which we assume now), we can examine the measure-valued Poisson equation  \eqref{eq:bihame} more closely, using the following result about equivalence between conceptions of solutions to a measure-valued Poisson problem with Dirichlet boundary conditions. For a comprehensive study of measure-valued Poisson equations we refer to \cite{Ponce}.
\begin{lemma}[Measure-valued Poisson equation, cf. {\cite[Proposition 6.3]{Ponce}} and {\cite[Proposition 5.1]{Ponce}}]\label{lem:maeval}
Suppose that $\Omega \subset \mathbb{R}^n$ is a bounded domain with smooth boundary and suppose that $\mu$ is a finite Radon measure on $\Omega$. Further, let $w : \Omega \rightarrow \overline{\mathbb{R}}$ be Lebesgue measurable.  Then the following are equivalent 
\begin{enumerate}
\item (Weak solutions with vanishing trace) 
\begin{equation*}
w \in W_0^{1,1}(\Omega) \quad    \; \textrm{and} \quad \; \int_\Omega \nabla w \nabla \phi \; \mathrm{d}x = \int_\Omega \phi \; \mathrm{d}\mu \quad \quad \forall \phi \in C_0^\infty(\Omega) .
\end{equation*}
\item (Test functions that can feel the boundary) 
\begin{equation*}
w \in L^1(\Omega) \quad  \;  \textrm{and} \quad \; - \int_\Omega w \Delta \phi \; \mathrm{d}x = \int_\Omega \phi \; \mathrm{d}\mu \quad \forall \phi \in C^\infty(\overline{\Omega}) : \phi_{\mid_{\partial\Omega}} \equiv 0 .
\end{equation*}
\end{enumerate}
If one of the two statements hold true, then $w \in W_0^{1,q}(\Omega)$ for each $q \in \left[1, \frac{n}{n-1}\right)$.  
\end{lemma}
This gives immediately the following 
\begin{cor} [Navier Boundary Conditions in the Trace Sense]
Suppose that $\Omega\subset \mathbb{R}^2$ has smooth boundary and let $u \in \mathcal{A}(u_0)$ be a minimizer. Then for each $\beta \in (0,1) $ one has that $u \in C^2(\Omega)\cap W^{3,2-\beta}(\Omega)$  and $\Delta u \in W_0^{1,2-\beta}(\Omega)$ .
\end{cor}
\begin{proof} Let $\beta \in ( 0 ,1)$. 
In view of \eqref{eq:bihame} one has that $w := \Delta u $ satisfies point $(2)$ of  Lemma \ref{lem:maeval} with $\mu = \frac{1}{2|\nabla u|} \mathcal{H}^1 \llcorner_{ \{u = 0 \} } $, which is a finite Radon measure because of Theorem \ref{thm:1.1}. We infer from Lemma \ref{lem:maeval} that $\Delta u \in W_0^{1,2-\beta} (\Omega)$. Since $2- \beta > 1$ we have maximal regularity for $\Delta u$ and can infer that $u \in W^{3, 2-\beta}(\Omega)$, cf. \cite[Theorem 9.19]{Gilbarg}. 
\end{proof}
\begin{remark}
Note in particular that the prevoius Corollary improves the regularity asserted in Theorem \ref{thm:1.1} for smooth domains $\Omega$. 
\end{remark}

We have shown that $\Delta u$ vanishes for a minimizer in the sense of traces. 
If $\Omega = B_1(0)$ there is another possible - and equally useful -  conception of vanishing at the boundary, namely that $\Delta u$ has vanishing radial limits on $\partial B_1(0)$, i.e. $\lim_{r \rightarrow 1-} \Delta u(r,\theta) = 0$ for a.e. $\theta \in (0,2\pi)$.

These two conceptions of vanishing have a nontrivial relation. A result that relates the concepts uses the fine topology, cf. \cite[Theorem 2.147]{ZiemerMaly}. We believe that consistency results can be shown with the cited theorem but the details would go beyond the scope of this article. Instead we give a self-contained proof that the Laplacian of a minimizer $u$ has vanishing radial limits in Appendix \ref{sec:vanrad}.

\section{Radial symmetry and Explicit Solutions}\label{sec:talenti}
In this section we show that for $\Omega = B_1(0)$ and $u_0 \equiv C$, there exists a radial minimizer. We will then be able to compute radial minimizers explicitly and determine the nonuniqueness level  $\iota$ from Definition   \ref{def:cand}.

\begin{definition}[Symmetric Nonincreasing Rearrangement]
Let $u : B_1(0) \rightarrow \overline{\mathbb{R}}$ be measurable. The function $u^* : B_1(0) \rightarrow \overline{\mathbb{R}}$ is the unique radial and radially nonincreasing function such that 
\begin{equation*}
| \{ u > t\} | = | \{ u^* > t \} | \quad \forall t \in \mathbb{R}.
\end{equation*}
\end{definition}
\begin{remark}\label{rem:102}
The fact that such a function exists follows from the construction in \cite[Chapter 3.3]{Lieb}. Moreover, one can show that for each $p \in [1,\infty]$,  $u \in L^p(B_1(0))$ implies that $u^* \in L^p(B_1(0))$ and $||u||_{L^p} = ||u^*||_{L^p}$, see \cite[Chapter 3.3]{Lieb}.
\end{remark}
We recall the famous Talenti rearrangement inequality, which we will use.
\begin{theorem}[Talenti's Inequality, cf. {\cite[Theorem 1]{Talenti}}]
Let $f \in L^2(B_1(0))$ and $w \in W_0^{1,2}(B_1(0))$ be the weak solution of 
\begin{equation*}
\begin{cases}
- \Delta u = f & \textrm{in } B_1(0) \\ u = 0 & \textrm{on } \partial B_1(0) 
\end{cases}.
\end{equation*}
Further, let $u\in W_0^{1,2}(B_1(0))$ be the weak solution of 
\begin{equation*}
\begin{cases}
- \Delta w = f^* & \textrm{in } B_1(0) \\ w = 0 & \textrm{on } \partial B_1(0) 
\end{cases}.
\end{equation*}
Then $w \geq u^*$ pointwise almost everywhere. 
\end{theorem}
We obtain the radiality of the solution as an immediate consequence. 
\begin{cor}[Radiality]\label{cor:radial}
Suppose that $\Omega= B_1(0)$ and $u_0 \equiv C$. Then there exists a minimizer $ v \in \mathcal{A}(C)$ that is radial. 
\end{cor}
\begin{proof}
First, fix a minimizer $u \in \A(C)$. Then by Remark \ref{rem:102} 
\begin{align}
\E(u) & = \int_\Omega (\Delta u )^2 \; \mathrm{d}x + |\{ u > 0 \}|  \nonumber \\
 & = \int_\Omega (\Delta (u- C))^2 \; \mathrm{d}x +  | \{ C - u <  C\} | \nonumber
 \\ & = \int_\Omega [( \Delta (u-C) )^* ]^2 \; \mathrm{d}x + |\{ (C-u)^* < C \}| \label{eq:tally} . 
\end{align}
Now define $w\in W^{2,2}(B_1(0)) \cap W_0^{1,2}(B_1(0))$ to be the weak solution of 
\begin{equation*}
\begin{cases}
- \Delta w = (\Delta (u-C))^* & \textrm{in } B_1(0) \\ w = 0 & \textrm{on } \partial B_1(0) 
\end{cases}.
\end{equation*}
Note that $w$ is radial since the right hand side is radial.  Observe now that $C- u$ is the unique weak solution of 
\begin{equation*}
\begin{cases}
- \Delta v = \Delta (u-C) & \textrm{in } B_1(0) \\ v = 0 & \textrm{on } \partial B_1(0) 
\end{cases}.
\end{equation*}
By Talenti's inequality (see previous theorem) $w \geq (C-u)^*$. In particular $|\{w < C\}| \leq |\{ (C-u)^* < C \} |$.  Therefore we can estimate  \eqref{eq:tally} in the following way
\begin{align*}
\E(u) & \geq \int_\Omega (\Delta w)^2 + | \{ w < C \} | \\ & = \int_\Omega (\Delta (C-w) )^2 - |\{ C- w > 0 \}| = \E(C-w).
\end{align*}  
Now define $v := C-w$. Then $v \in \mathcal{A}(C)$ since $v- C = -w \in W_0^{1,2}(\Omega)$. By the estimate above we see that $v$ is yet another minimizer.  
\end{proof}
Now we characterize the radial solutions explicitly using the following two propositions 
\begin{prop}[Radial Solutions on Annuli]
Let $A_{R_1,R_2} := \{ x \in \mathbb{R}^2 : R_1 < |x| < R_2 \}$ be an annulus with inner radius $R_1 \geq 0 $ and outer radius $R_2 > R_1$. If $w \in W^{2,2}(A_{R_1,R_2})$ is weakly biharmonic and radial then there exists constants $A,B,C, D \in \mathbb{R} $ such that 
\begin{equation}\label{eq:bihamsol}
w(x) = A|x|^2 + B + C \log|x|  + D \frac{|x|^2}{2} \log|x| 
\end{equation}
\end{prop}
\begin{proof}
The claim reduces to a straightforward ODE argument when expressing $\Delta^2$ in polar coordinates. 
\end{proof}

\begin{prop}\label{prop:radsub}[Radial Zero Level Set] 
Let $u \in \A(u_0)$ be a radial minimizer. Then there exists $R_0 > 0 $ such that 
\begin{equation}
\{ u = 0 \} = \partial B_{R_0}(0)
\end{equation}
and $\{ u  > 0 \} = B_1(0) \setminus \overline{B_{R_0}(0)}$. 
\end{prop}
\begin{proof}
According to Theorem \ref{thm:1.1} one has $\{u = 0 \} = \bigcup_{i = 1}^N S_i $ for closed disjoint $C^2$-manifolds $S_i$ all of which form a connected component of $\{ u = 0 \}$. Since $u$ is radial one has $S_i = \partial  B_{r_i}(0)$ for some radii $r_i > 0 $. Without loss of generality $r_1 < ... < r_N$. It remains to show that $N = 1$. If $N >1$ then $u \equiv 0 $ on $\partial B_{r_N}(0)$. By subharmonicity one has $u <0 $ on $B_{r_N}(0)$. However now $r_1< r_N$ and therefore one obtains a contradiction to $u = 0 $ on $\partial B_{r_1}(0)$.
\end{proof}

\begin{lemma} [Explicit radial solutions] \label{lem:explrad}
Suppose $\Omega = B_1(0)$. Let $u_0$ be a positive constant. Define  
\begin{equation}\label{eq:infii}
h(u_0) := \min \left\lbrace \pi , \inf_{R_0 \in (0,1)} \left( \frac{4\pi u_0^2}{\frac{1-R_0^2}{2}+ R_0^2 \log(R_0) } + \pi (1-R_0^2) \right) \right\rbrace .
\end{equation}
In case that $h(u_0) < \pi$, the infimum in \eqref{eq:infii} is attained and for each $R \in (0,1)$ that realizes the infimum in \eqref{eq:infii} the function 

\begin{equation*}
u(x) =u_0\begin{cases} \frac{  \log R  |x|^2 - R^2 \log R}{R^2-1- 2R^2 \log R}  & 0 \leq |x| \leq  R, \\ \frac{-  |x|^2 + R^2 -2 R^2 \log R  +   R^2 \log |x| +  {|x|^2} \log |x|}{{R^2-1}-2 R^2 \log R} & R < |x| < 1. \end{cases}
\end{equation*} 
is a minimizer with energy $\E(u) = h(u_0)$. In case that $h(u_0) = \pi$ a minimizer is given by a constant. 
\end{lemma}
\begin{proof}Recall there exists a radial minimizer $u$  by Corollary \ref{cor:radial}. 
By Theorem \ref{thm:1.1}, Corollary \ref{lem:corsubham} and Proposition \ref{prop:radsub} we deduce that $\{ u = 0 \} $ is either empty or there exists $R_0 \in (0,1)$ such that $\{ u = 0 \} = \partial B_{R_0}(0)$. If $\{u = 0\}$ is empty then the minimizer is a contant.  In the other case, Lemma \ref{lem:biham} implies that $u$ is weakly biharmonic on the annuli $\{ 0 < |x| < R_0 \}$ and $\{R_0 < |x|<1\}$. Hence there exist real numbers $C_1,D_1, E_1,F_1,C_2,D_2,E_2,F_2$ such that 
\begin{equation*}
u(x)  = \begin{cases}  C_1 |x|^2 + D_1 + E_1 \log|x| + F_1 \frac{|x|^2}{2}\log|x| & 0 < |x| < R_0 \\ C_2 |x|^2 + D_2 + E_2 \log |x| + F_2 \frac{|x|^2}{2} \log |x| & R_0 < |x| < 1 \end{cases} .
\end{equation*} 
Since $u$ has to be continuous at zero we deduce that $E_1 = 0 $. Since second derivatives of $u$ have to be continuous at zero it is an easy computation to show that $F_1 = 0$. By the Navier boundary conditions (cf. Appendix \ref{sec:vanrad}) we get that $4C_2 + 2F_2 = 0 $ and thus
\begin{equation}\label{eq:hilfreich}
\Delta u(x) =  \begin{cases} 4C_1   & 0 < |x| < R_0 \\2 F_2 \log|x|  & R_0 < |x| < 1 \end{cases}. 
\end{equation}
As $\Delta u$ has to be continuous we obtain that $4C_1 = 2 F_2 \log R_0 $, i.e. $C_1 = \frac{1}{2}F_2 \log R_0$.
The fact that $ u = 0 $ on $\partial B_{R_0} (0 )$ implies that 
$0 =  C_1 R_0^2 + D_1 $  and hence $D_1 = - C_1 R_0^2 = - \frac{F_2}{2}R_0^2 \log R_0$. From all these computations we obtain 
\begin{equation*}
u(x)  = \begin{cases} \frac{1}{2} F_2 \log R_0  |x|^2 - \frac{1}{2}F_2 R_0^2 \log R_0  & 0 < |x| \leq  R_0 \\- \frac{1}{2} F_2 |x|^2 + D_2 + E_2 \log |x| + F_2 \frac{|x|^2}{2} \log |x| & R_0 < |x| < 1 \end{cases} .
\end{equation*}
If we take the radial derivative $\partial_r u$ in both cases and set them equal we obtain 
\begin{equation*}
F_2 R_0  \log R_0 = - F_2 R_0 + E_2 \frac{1}{R_0} + F_2 R_0 \log R_0 + \frac{1}{2} F_2 R_0 ,
\end{equation*}
which results in $E_2 = \frac{1}{2}F_2 R_0^2$ and thus 
\begin{equation*}
u(x) =  \begin{cases} \frac{1}{2} F_2 \log R_0  |x|^2 - \frac{1}{2}F_2 R_0^2 \log R_0  & 0 < |x| \leq  R_0 \\- \frac{1}{2} F_2 |x|^2 + D_2 +  \frac{1}{2} F_2 R_0^2 \log |x| + F_2 \frac{|x|^2}{2} \log |x| & R_0 < |x| < 1 \end{cases} .
\end{equation*}
Note another time that $0 = \lim_{ |x| \rightarrow R_0 + } u$ and therefore 
\begin{equation*}
0 = D_2 + F_2 \left( - \frac{1}{2}R_0^2 + R_0^2 \log R_0  \right) .
\end{equation*}
Hence $D_2 = \frac{1}{2}F_2 R_0^2 - F_2 R_0^2 \log R_0 $ and this yields that
\begin{equation*}
u(x) = F_2 \begin{cases} \frac{1}{2}  \log R_0  |x|^2 - \frac{1}{2}R_0^2 \log R_0  & 0 \leq |x| \leq  R_0 \\- \frac{1}{2}  |x|^2 + \frac{1}{2} R_0^2 - R_0^2 \log R_0  +  \frac{1}{2} R_0^2 \log |x| +  \frac{|x|^2}{2} \log |x| & R_0 < |x| < 1 \end{cases}
\end{equation*} 
Using that $u  \equiv u_0$ on $\partial B_1(0)$ we find 
\begin{equation}\label{eq:eff2}
u_0 = F_2 \left( \frac{R_0^2-1}{2}- R_0^2 \log R_0 \right),
\end{equation}
which finally determines $R_0$. Hence we know that there must exist some $R_0 \in (0,1)$ such that 
\begin{equation}\label{eq:notkrituu}
u(x) =u_0\begin{cases} \frac{  \log R_0  |x|^2 - R_0^2 \log R_0 }{R_0^2-1- 2R_0^2 \log R_0}  & 0 \leq |x| \leq  R_0, \\ \frac{-  |x|^2 + R_0^2 -2 R_0^2 \log R_0  +   R_0^2 \log |x| +  {|x|^2} \log |x|}{{R_0^2-1}-2 R_0^2 \log R_0} & R_0 < |x| < 1. \end{cases}
\end{equation} 
Now define for $R_0 \in (0,1)$ the function $w_{R_0} \in \A(u_0)$ to be the right hand side of \eqref{eq:notkrituu}.
We have shown either $\{u = 0 \} $ is empty or that the minimizer is given by some $w_{R_0^*}$ for some $R_0^* \in (0,1)$. 
Going back to \eqref{eq:hilfreich} and using that according to Proposition \ref{prop:radsub} $|\{ u > 0 \}| = \pi (1-R_0^2)$ we obtain that 
\begin{align*}
\E(w_{R_0})& = 16C_1^2 \pi R_0^2 +4 F_2^2 \int_{B_1 \setminus B_{R_0}} 4 F_2^2 \log^2|x| \; \mathrm{d}x + \pi (1- R_0^2) \\&  = 4F_2^2 \pi \left( R_0^2 \log^2 R_0 + 2 \int_{R_0}^1 r \log^2 r \; \mathrm{d}r \right) + \pi (1- R_0^2), 
\end{align*} 
where we use the derived parameter identity for $C_1$ and radial integration in the last step. Using that 
\begin{equation*}
\int_{R_0}^1 r \log^2 r \; \mathrm{d}r = \frac{R_0^2}{2} \log^2 R_0 + \frac{R_0^2}{2} \log R_0 + \frac{1- R_0^2}{4}
\end{equation*}
we obtain using \eqref{eq:eff2}
\begin{align}
\E(w_{R_0} ) & = 4 F_2^2 \pi \left( R_0^2 \log R_0 + \frac{1-R_0^2}{2} \right) + \pi (1- R_0^2) \nonumber \\ & \label{eq:ero}  = \frac{4\pi u_0^2}{\frac{1-R_0^2}{2} + R_0^2 \log R_0 } + \pi (1- R_0^2) .
\end{align}
We have shown that for each $R_0 \in (0,1)$ we can find an admissible function $w_{R_0} \in \A(u_0)$ such that $\E(w_{R_0})$ is given by the right hand side of \eqref{eq:ero}. Moreover we know that a minimizer $u$ is among such $w_{R_0}$ in case that $\{ u = 0 \} \neq \emptyset$. In case that $\{u = 0 \} = \emptyset $ however, we know from Remark \ref{rem:nontriv} that $\E(u) = \pi$. We obtain that 
\begin{equation}\label{eq:enerrad}
\E(u) = \min \left\lbrace \pi,  \inf_{R_0 \in (0,1) }  \frac{4\pi u_0^2}{\frac{1-R_0^2}{2} + R_0^2 \log R_0 } + \pi (1- R_0^2) \right\rbrace .
\end{equation}
and in case that  the infimum is smaller than $\pi$, it is attained by some $R_0^* \in (0,1) $ such that a minimizer is given by $w_{R_0^*}$. 
\end{proof}
\begin{remark}
Let $h(u_0)$ be the quantity defined in the previous lemma. If $h(u_0)< \pi$ then one has to find  
\begin{equation*}
\inf_{R_0 \in (0,1)} \left( \frac{4\pi u_0^2}{\frac{1-R_0^2}{2}+ R_0^2 \log(R_0) } + \pi (1-R_0^2) \right).  
\end{equation*}
To do so, we set the first derivative of the expression equal zero, which becomes
\begin{equation*}
0 = \frac{-8\pi u_0^2 R_0 \log R_0 }{\left( \frac{1-R_0^2}{2}+ R_0^2 \log R_0 \right)^2} - 2\pi R_0^2.
\end{equation*} 
Solving for $u_0$ and plugging this into \eqref{eq:infii} we find that 
\begin{align}
\inf_{ w \in \mathcal{A}(u_0) }\E(w) & = - \frac{\pi}{\log R_0 } \left( \frac{1-R_0^2}{2}+ R_0^2 \log R_0 \right) + \pi (1- R_0^2) \nonumber
\\ & =  \pi \left( 1- 2R_0^2 + \frac{R_0^2-1}{2 \log R_0} \right) \label{eq:rdill}
\end{align}
\end{remark}

\begin{lemma}[The nonuniqueness level]\label{lem:iooo}
Let $\Omega = B_1(0)$. Then the quantity $\iota$ in Definition \ref{def:cand} is given by  
 \begin{equation}\label{eq:iooo}
 \iota = \frac{R_*}{2} \sqrt{\frac{1 -R_*^2}{2} + R_*^2 \log R_* } \simeq 0.112814
 \end{equation}
 where $R_* \simeq 0.533543$ is the unique solution of 
 \begin{equation}\label{eq:glfrr}
 \frac{R^2 - 1}{2\log R} - 2R^2 = 0, \quad R \in (0,1).
 \end{equation}
\end{lemma}
\begin{proof}
First we show that \eqref{eq:glfrr} has a unique solution $R_* \in (0,1)$. For this we first rewrite the equation multiplying by $2\log R$. 
\begin{equation*}
R^2 - 1 - 4 R^2 \log R = 0 
\end{equation*}
Using that $2 \log R = \log R^2 $ and substituting $R^2 = e^u$ for some $u \in (- \infty, 0) $ we find 
 \begin{equation*}
 e^u - 1 - 2ue^u = 0 \quad \Leftrightarrow \quad  ( 1-2u) = e^{-u}. 
 \end{equation*}
 By \cite[Eq. (2.23)]{Lambert} this equation has the solution 
 \begin{equation}\label{eq:spezfu}
 u \in  \frac{1}{2}+ W \left(- \frac{1}{2\sqrt{e}} \right) 
 \end{equation}
 where $W$ denotes the Lambert $W-$function, i.e. the multi-valued inverse of $f(x) = x e^x$. Note that for each negative number in $a \in (- e^{-1},0)$, $W(a)$ is exactly two-valued with one value smaller than $-1$ and one value larger than $-1$. This can be seen using that $f$ is negative on $(- \infty, 0)$ and has a global minimum at $-1$ with value $e^{-1}$. Moreover $f$ is decreasing on $(-\infty, -1) $ and increasing $(-1, 0)$. All of these assertions can be proved with standard techniques. 
 Now note that $f(-\frac{1}{2}) = - \frac{1}{2\sqrt{e}}$ and therefore $-\frac{1}{2}$ is one values of $w$, i.e. the first possible solution of $u$ is $u= 0 $. This however does not lie in ths interval $(-\infty, 0 )$ and hence resubstitution does not generate a vlues $R \in (0,1)$. The only remaining possibility is the other value of $W \left(- \frac{1}{2\sqrt{e}} \right)$ that falls strictly below $-1$ and hence the corresponding solution for $u$ lies in $(- \infty, 0 )$, cf. \eqref{eq:spezfu}. Therefore this unique solution $u_* \in (-\infty, 0 )$ generates a unique solution $R_* = e^{\frac{1}{2} u_*} \in (0,1)$. Now we show \eqref{eq:iooo}. 

By Lemma \ref{lem:limica} one can find a minimizer with energy $\pi= |B_1(0)|$. Recall from the proof of Theorem \ref{thm:nonun} that a minimizer $u \in \mathcal{A}(\iota)$ can be constructed by taking a weak $W^{2,2}$-limit of  minimizers $u_n \in \mathcal{A}(\iota_n)$ for some sequence of constants $(\iota_n)_{n \in \mathbb{N}}$ that converges from below to $\iota$. Without loss of generality we can assume that there exists $\delta > 0 $ such that $\iota_n \geq \delta > 0 $ for each $n \in \mathbb{N}$. By definition of $\iota$ one can  achieve that $\E(u_n) < \pi$ for all $n \in \mathbb{N}$. Repeating the computation in \eqref{eq:linenerg} one can also has
\begin{equation*}
\pi = \lim_{n \rightarrow \infty } \E(u_n).
\end{equation*}
Now note also that $(u_n)_{n \in \mathbb{N}}$ can be chosen to be a sequence of radial minimizers. In particular we can choose $u_n$ to be of the form 
\begin{equation*}
u(x) = \iota_n \begin{cases} \frac{  \log R_n  |x|^2 - R_n^2 \log R_n }{R_n^2-1- 2R_n^2 \log R_n}  & 0 \leq |x| \leq  R_n, \\ \frac{-  |x|^2 + R_n^2 -2 R_n^2 \log R_n  +   R_n^2 \log |x| +  {|x|^2} \log |x|}{{R_n^2-1}-2 R_n^2 \log R_n} & R_n < |x| < 1 \end{cases}
\end{equation*} 
for some $R_n \in (0,1)$. By \eqref{eq:rdill} we infer that $R_n$ satisfies
\begin{equation*}
 \pi \underset{n \rightarrow \infty}{\longleftarrow} \E(u_n) = \pi \left( 1 - 2R_n^2 + \frac{R_n^2 - 1}{2\log(R_n)} \right)
\end{equation*}
and hence 
\begin{equation*}
\frac{R_n^2 - 1}{2\log(R_n)} - 2R_n^2 \rightarrow 0 \quad ( n \rightarrow \infty) .
\end{equation*}
By \eqref{eq:negativityset}, we obtain that $\pi R_n^2 \geq 4\pi \iota_n^2 \geq 4\pi^2 \delta$. Therefore $R_n \geq 2\sqrt{\pi }\sqrt{\delta}$ is bounded from below by a strictly positive constant. 
Define $a : [0,1] \rightarrow \mathbb{R}^2$ to be the continuous extension of $z \mapsto \frac{z^2-1}{2\log z} - 2 z^2$. 
By compactness of $[0,1]$, $(R_n)_{n \in \mathbb{N}}$ has a convergent subsequence (again denoted by $(R_n)_{n \in \mathbb{N}}$)to some limit $R \in [0,1]$ that satisfies $a(R) = 0 $. Since $a(1) = -3 \neq 0 $ this equation is only solved by zero and by $R_*$ determined above. However $R \neq 0 $ since $(R_n)_{n \in \mathbb{N}}$ is bounded away from zero. This implies that $R = R_*$ and in particular that $(R_n)_{n \in \mathbb{N}}$ converges to $R_*$. By Lemma \ref{lem:explrad} we infer - since $\E(u_n) < \pi$ -  that 
\begin{equation*}
\E(u_n) = \frac{4\pi \iota_n^2}{\frac{1-R_n}{2}+ R_n^2 \log R_n} + \pi ( 1- R_n^2).
\end{equation*}
Using that $\iota_n \rightarrow \iota$, $R_n \rightarrow R_*$ and $\E(u_n) \rightarrow \pi$ as $n \rightarrow \infty$ we obtain in the limit that 
\begin{equation*}
\pi = \frac{4\pi \iota^2}{\frac{1-R_*}{2}+ R_*^2 \log R_*} + \pi ( 1- R_*^2) .
\end{equation*} 
Solving for $\iota$ we obtain the claim. 
\end{proof}

Next we list some selected numerical values for radial minimizers in Table \ref{tabelle} and give some plots in the figure below. For this let $R$ be the set of all points $R_0 \in (0,1)$ where $h(u_0)$ in \eqref{eq:infii} is attained (which conicides with the radius of the nodal sphere of a minimizer)

\begin{table}[h]
\begin{tabular}{|c|c|c|}
\hline 
$u_0$ &$R(u_0)$   & $\inf_{A(u_0)} \E$ \\
\hline 
$0.01$ & $0.924036$ & $0.682707$ \\
\hline
$0.02$ & $0.876984$ & $1.07223$ \\
\hline
$0.04$ & $0.797621$ & $1.67144$ \\ 
\hline 
$0.08$ & $0.654679$ & $2.56739$ \\
\hline 
$0.1$ & $0.582373$ & $2.93062$ \\ 
\hline 
$0.11$ & $0.544514$ & $3.09661$ \\
\hline 
$0.112$ & $0.536733$ & $3.12866$ \\
\hline  
\end{tabular}
\vspace{0.1cm}
\caption{Energy and nodal radius for selected boundary data}\label{tabelle}
\end{table}

 \begin{figure}[ht]
    \centering
    \begin{subfigure}[b]{0.48\textwidth}
        \includegraphics[width=\textwidth]{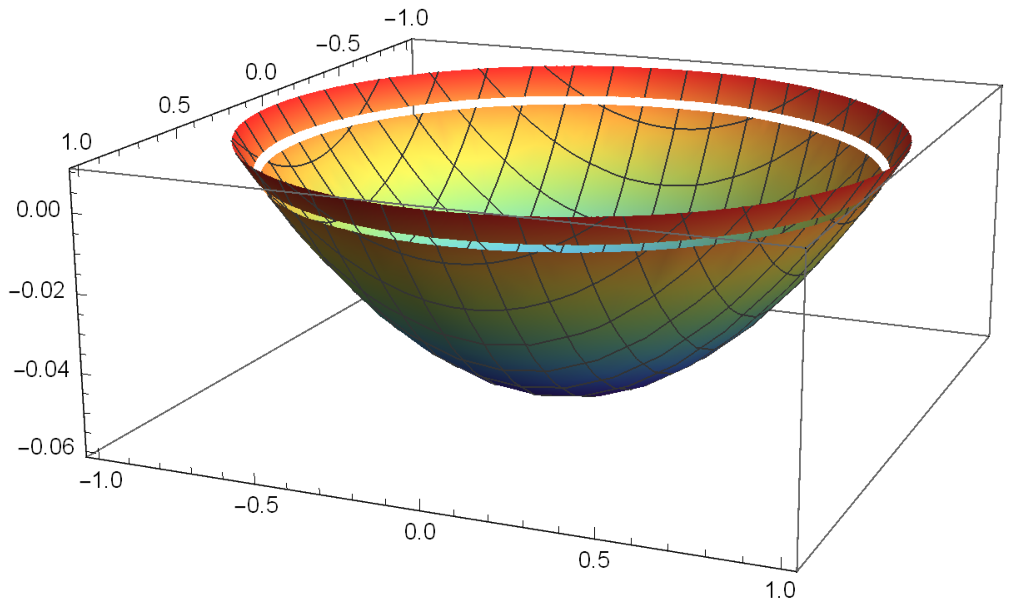}
        \caption{$u_0 = 0.01$: 3D-Plot}
    \end{subfigure}
    ~
    \begin{subfigure}[b]{0.48\textwidth}
        \includegraphics[width=\textwidth]{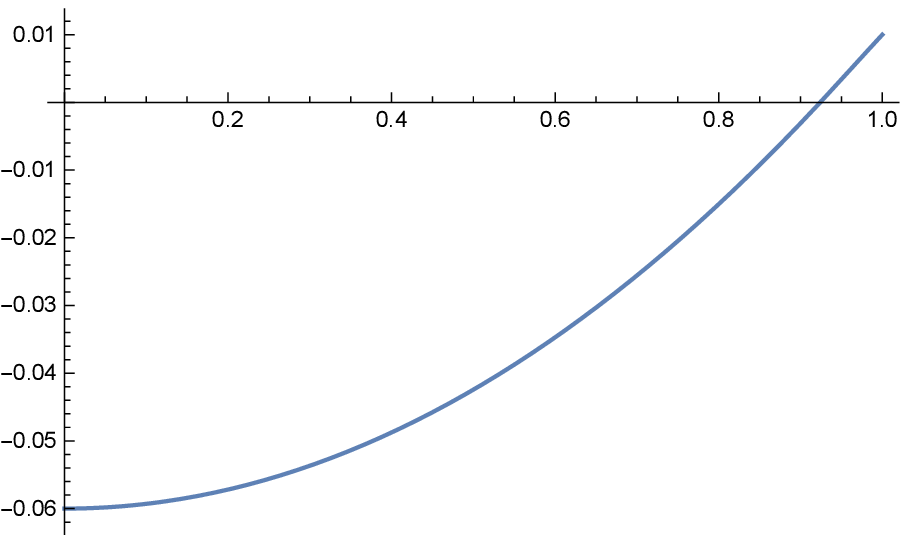}
        \caption{$u_0 = 0.01$: Profile curve}
    \end{subfigure}
    
    \begin{subfigure}[b]{0.48\textwidth}
        \includegraphics[width=\textwidth]{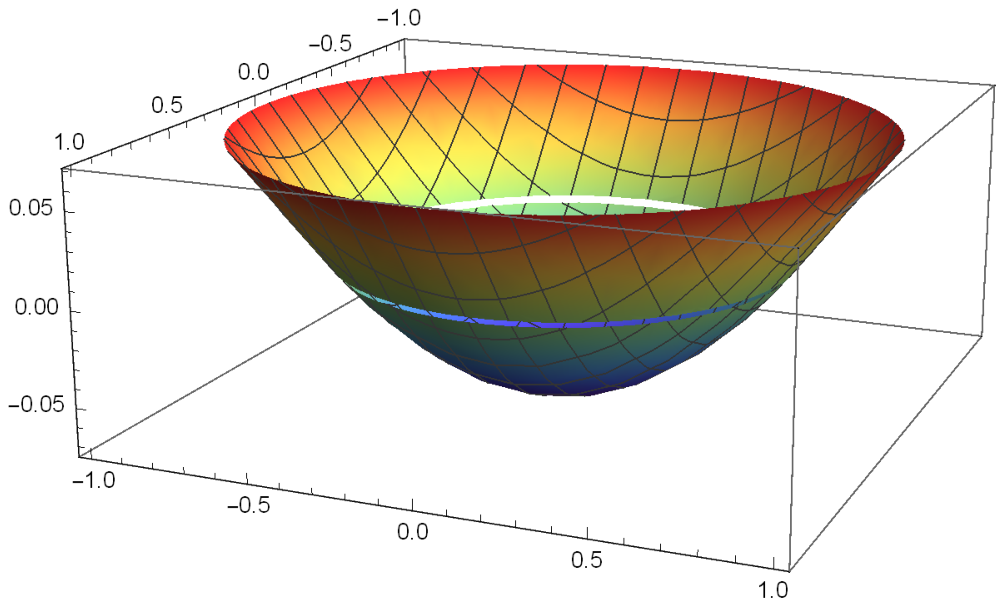}
        \caption{$u_0 = 0.07$: 3D-Plot}
    \end{subfigure}
    ~
    \begin{subfigure}[b]{0.48\textwidth}
        \includegraphics[width=\textwidth]{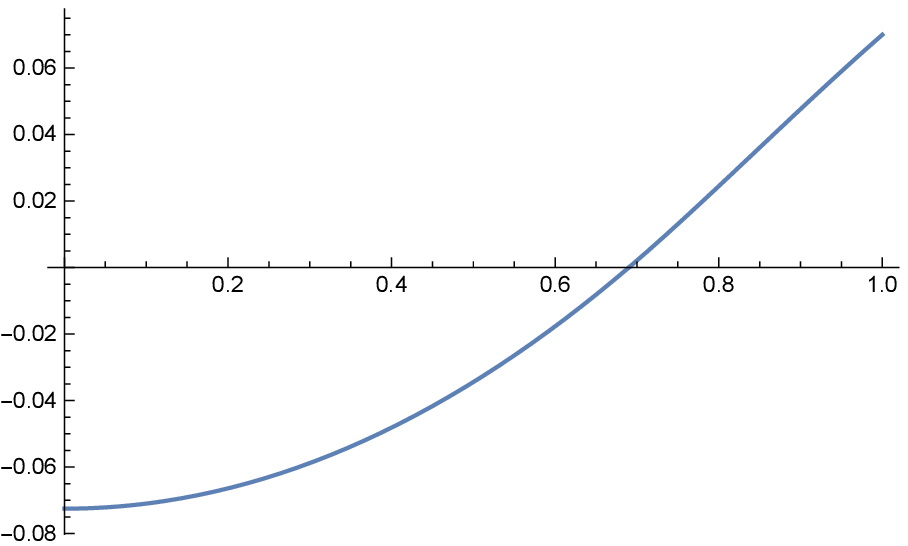}
        \caption{$u_0= 0.07$: Profile curve}
    \end{subfigure}  
    
     \begin{subfigure}[b]{0.48\textwidth}
        \includegraphics[width=\textwidth]{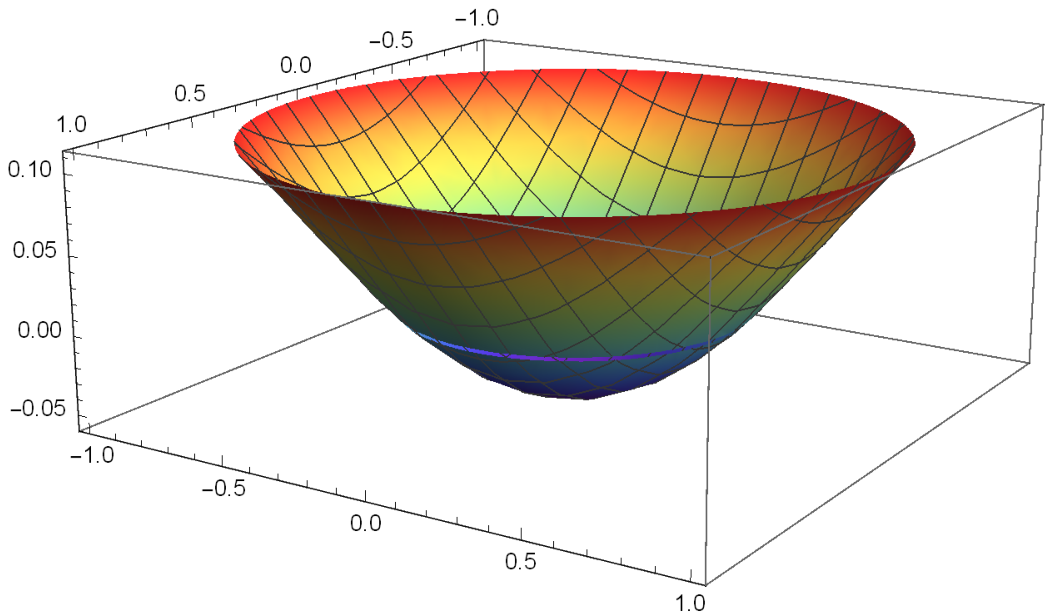}
        \caption{$u_0= 0.112$: 3D-Plot}
    \end{subfigure}
    ~
    \begin{subfigure}[b]{0.48\textwidth}
        \includegraphics[width=\textwidth]{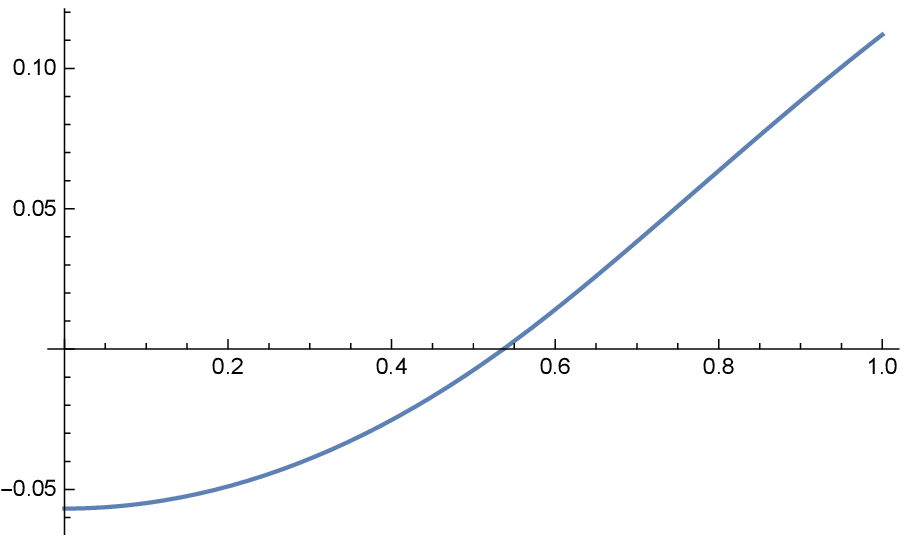}
        \caption{$u_0 = 0.112$: Profile curve}
    \end{subfigure}  
    \caption{Selected Minimizers and their radial profile curves}\label{fig:OrbitlikeFreeElastica}
\end{figure}

\section{Optimality Discussion and Closing Remarks}\label{sec:opti}
In this section we present some open problems which we think would be interesting to consider in the context of the biharmonic Alt-Caffarelli problem. As we outlined in the introduction, the biharmonic Alt-Caffarelli problem is fundamentally different from some more established higher order variational problems with free boundary and therefore we believe that new techniques have to be developed.

\begin{remark}[Interior  regularity]
It is an interesting question whether one can expect more interior regularity of minimizers than $C^2(\Omega)$. Recall that by Theorem \ref{thm:1.1}, regularity of minimizers and regularity of the free boundary are connected by the fact that minimizers have nonvanishing gradient on their nodal set. There is however one obstruction to higher regularity: The explicit minimizers found in Lemma \ref{lem:explrad} are do not lie in $C^3(\Omega)$. What remains then open is $C^{2,\gamma}$-regularity for some $\gamma > 0$. The solutions in Lemma \ref{lem:explrad} actually have a Lipschitz second derivative, so it is likely that better regularity statements can be derived. 
\end{remark}

Another interesting and not immediately related question is the regularity up to the free boundary.

\begin{remark}[Regularity up to the free boundary]
We have found in Lemma \ref{lem:biham} that $u \in C^\infty( \{ u> 0 \})$, as $\Delta u $ is harmonic in $\{ u > 0 \}$. Moreover, Theorem \ref{thm:1.1} implies that $\Delta u $ is continuous on $\overline{\{ u < 0 \} }$, which makes it a classical solution to a Dirichlet problem.   
Higher Regularity of $\Delta u_{\mid_{ u< 0 }}$ up to the free boundary turns out to be an interesting problem. The free boundary is regular enough for elliptic regularity theory, cf. \cite[Theorem 9.15]{Gilbarg}. However, it is unclear whether $\Delta u_{\mid_{ \{ u = 0 \}}}$ is a trace of a $W^{2,p}$ function for any $p \in (1,2)$, which is also a requirement in \cite[Theorem 9.15]{Gilbarg}.
 This is actually delicate, see  \cite{Hadamard} and \cite[Page 3]{Daners} for relevant counterexamples. Further regularity up to the free boundary would improve the regularity of the free boundary itself, which we do not consider impossible. Hence such a discussion is useful and could potentially give way to future research. 
\end{remark}

\begin{remark}[Dirichlet boundary conditions]
The argumentation in the present article relied heavily on a weak version of the  `maximum principle for systems', see \cite[Section 2.7]{Sweers} for the exact connection between elliptic systems and higher order PDE's with Navier boundary conditions. In the case of Dirichlet boundary conditions where these priciples are not available, statements like Corollary \ref{lem:corsubham} are not expected to hold true in the way they do in our case. We nevertheless believe that a discussion of the Dirichlet problem is also doable. It would be an interesting question whether the results are similar at all. The question has also been asked for other higher order free boundary problems, see \cite{Friedman} for the biharmonic obstacle problem, where rich similarities can be found.  
\end{remark}

\begin{remark}[Connectedness of the Free Boundary] 
It would also be interesting to understand some global properties of the minimizer. For example it is worth asking whether conditions on $\Omega$ and $u_0$ can be found under which the free boundary is connected, i.e. $N = 1$ in Theorem \ref{thm:1.1}. One would expect that $N=1$ is not always the case, for example for dumbbell-shaped domains. Such global properties of the solution are difficult to understand - again due to the lack of a maximum principle for fourth order equations. 
\end{remark}

\appendix

\section{Proof of Lemma \ref{lem:bihammeasrep}} \label{sec:tecprf}
\begin{proof}[Proof of Lemma \ref{lem:bihammeasrep}]
What one formally does in this proof is plug the fundamental solution $F$ into \eqref{eq:bihammeas}. Since $F$ is however not admissible for this equation one has to localize and regularize before, producing error terms that we study in the following. 

Choose $\xi \in C_0^\infty(\Omega)$ such that $0 \leq \xi \leq 1$ and $\xi \equiv 1$ on an open neighborhood $U$ of $\Omega_{\epsilon_0}^C$. Now choose $(\zeta_n)_{n = 1}^\infty \subset C_0^\infty(\Omega)$ such that $\zeta_n \rightarrow \Delta (u-u_0) $ in $L^2(\Omega)$. For each $n \in \mathbb{N}$ let $\phi_n$ be the solution of 
\begin{equation*}
\begin{cases}
\Delta \phi_n = \zeta_n  & \mathrm{in} \; \Omega \\
\phi_n = 0 &  \mathrm{on} \; \partial \Omega . 
\end{cases}
\end{equation*} 
By elliptic regularity, $\phi_n \in W_0^{1,2}(\Omega)\cap W^{2,2}(\Omega)\cap C^\infty(\Omega)$. Moreover, by Definition \ref{def:hilbers}
\begin{equation*}
 ||\phi_n - (u-u_0)||_{W^{2,2}\cap W_0^{1,2}} = ||\Delta \phi_n - \Delta (u-u_0)||_{L^2} =||\zeta_n - \Delta(u-u_0)|| \rightarrow 0  \quad (n \rightarrow \infty) .
\end{equation*} 
 Since $W^{2,2}(\Omega) $ embeds into $C(\overline{\Omega})$, we obtain that $\phi_n$ converges uniformly to $u - u_0$.  Since $\phi_n \xi \in C_0^\infty(\mathbb{R}^2)$  we can compute 
\begin{align}
u(x) \xi(x) & = (u(x) - u_0(x) )\xi(x) + u_0(x) \xi(x) \nonumber
\\ & = \lim_{n \rightarrow \infty} \phi_n(x) \xi(x) + u_0(x) \xi(x) \nonumber
\\ & = \lim_{n \rightarrow \infty} \int F(x,y) \Delta^2 ( \phi_n(y) \xi(y) ) dy + u_0(x) \xi(x) \nonumber
\\ &= \lim_{n \rightarrow \infty} \int \Delta_y F(x,y) \Delta ( \phi_n(y) \xi(y) ) \dy + u_0(x) \xi(x) \nonumber
\\ & = \lim_{n \rightarrow \infty} \int_\Omega \Delta_y F(x,y) ( \Delta \phi_n (y) \xi(y) + 2 \nabla \phi_n (y) \nabla \xi(y) + \phi_n(y) \Delta \xi(y) )  \dy \nonumber
 \\ & \qquad \qquad  + u_0(x) \xi(x) \nonumber
\\ & = \int_\Omega \Delta_y F(x,y) \Delta(u-u_0)(y) \xi(y) \dy   + 2\int_\Omega \Delta_y F(x,y)\nabla(u-u_0)(y)  \nabla \xi(y) \dy \nonumber \\ & \qquad \qquad  + \int_\Omega \Delta_y F(x,y) (u-u_0)(y) \Delta \xi(y) \dy + u_0(x) \xi(x) \nonumber
\\ & = \int_\Omega \Delta_y F(x,y) \Delta u (y) \xi(y) \dy - \int_\Omega  \Delta_y F(x,y) \Delta u_0(y) \xi(y) \dy + R_1(x), \label{eq:223} 
\end{align}
where 
\begin{align}\label{eq:defR1}
R_1(x)  & :=  2\int_\Omega \Delta_y F(x,y)\nabla(u-u_0)(y)  \nabla \xi(y) \dy \nonumber   \\ & \quad \quad + \int_\Omega \Delta_y F(x,y) (u-u_0)(y) \Delta \xi(y) \dy  + u_0(x) \xi(x). 
\end{align}
Since $F$ is smooth as long as $x \neq y$ and by choice of $\xi$ we obtain that $R_1 \in C^\infty( \overline{\Omega_{\epsilon_0}^C})$.
We further examine \eqref{eq:223} noting that
\begin{equation}\label{eq:smoothi}
\int_\Omega \Delta_y F(x,y) \Delta u_0(y) \xi(y) \dy = \int_\Omega \Delta_y F(x,y) \Delta (u_0 \xi)(y)  \dy + R_2(x)  
\end{equation}
where 
\begin{equation*}
R_2(x) := - 2 \int_\Omega\Delta_yF(x,y) \nabla u_0(y) \nabla \xi(y) \dy - \int_\Omega \Delta_y F(x,y) u_0(y) \Delta \xi(y) \dy,   
\end{equation*}
which lies in $C^\infty(\overline{\Omega_{\epsilon_0}^C})$  with the same arguments used above. Now since $u_0 \xi \in C_0^\infty(\Omega)\subset C_0^\infty(\mathbb{R}^2)$, \eqref{eq:smoothi} simplifies by definition of the fundamental solution to 
\begin{equation*}
\int_\Omega \Delta_y F(x,y) \Delta u_0(y) \xi(y) \dy = u_0(x) \xi(x) + R_2(x) , 
\end{equation*}
and the right hand side of this equation lies in $C^\infty( \overline{\Omega_{\epsilon_0}^C}) $. Using this and \eqref{eq:223} we obtain 
\begin{equation}\label{eq:228}
u(x) \xi(x) = \int_\Omega \Delta_y F(x,y) \Delta u (y) \xi(y) \dy + \widetilde{h}(x)
\end{equation}
for some $\widetilde{h} \in C^\infty(\overline{\Omega_{\epsilon_0}^C})$. Now observe that 
\begin{equation}\label{eq:229}
 \int_\Omega \Delta_y F(x,y) \Delta u (y) \xi(y) dy = \int_\Omega \Delta u \Delta_y ( F(x, \cdot) \xi(\cdot)) dy + R_3(x) 
\end{equation}
where 
\begin{equation*}
R_3(x) = - 2 \int_\Omega \Delta u(y) \nabla_y F(x,y) \nabla \xi(y) dy - \int_\Omega \Delta u (y) F(x,y) \Delta \xi(y) dy .
\end{equation*}
Note that $R_3 \in C^\infty( \overline{\Omega_{\epsilon_0}^C} )$ for the very same reason as $R_1,R_2 $ are $C^\infty$. For the first summand on the right hand side of \eqref{eq:229} we can use \eqref{eq:bihammeas} since $F(x, \cdot) \xi \in W_0^{2,2}(\Omega)$ to obtain 
\begin{equation}\label{eq:regula}
u(x) \xi(x) = - \frac{1}{2}\int_\Omega F(x,y) \xi(y) d\mu(y)  + h(x) 
\end{equation}
for some $h \in C^\infty(\overline{\Omega_{\epsilon_0}^C} ) $. Note that by construction of $\Omega_{\epsilon_0}$ in  Corollary \ref{cor:guterrand}, $\xi \equiv 1 $ on $U \supset \Omega_{\epsilon_0}^C \supset \{ u = 0\} \supset \supp(\mu)$. Therefore we can leave out the $\xi$ in the $\mu$-integration. Now we plug in $x \in \overline{\Omega_{\epsilon_0}^C}$. Then $\xi(x) = 1$ and hence by  \eqref{eq:regula} 
\begin{equation*}
u(x) = - \frac{1}{2}\int_\Omega F(x,y) d\mu(y) + h(x). \qedhere
\end{equation*}
\end{proof}

\section{Proof of Lemma \ref{lem:regunod}} \label{sec:regunod}

\begin{proof}[Proof of Lemma \ref{lem:regunod}]
We only show, for the sake of simplicity that $\Delta u \in C(\Omega)$. Other second derivatives work similarly. 
For $x \in \{ u> 0 \} \cup \{ u< 0 \}$ one can infer continuity of the Laplacian from Lemma \ref{lem:biham}. Now fix $x \in \{ u = 0 \} $. We show that $\Delta u $ is continuous at $x$. Choose $r \in( 0, \tfrac{1}{2})  $ such that $\nabla u \neq 0 $ on $\overline{B_{2r}(x)}$ and $B_r(x) \cap \{ u = 0 \} $ possesses a graph representation i.e. there exists a bounded interval $U \supset \{ x_1\} $ open and $h \in C^1(\overline{U})$ such that 
\begin{equation}\label{eq:graphr}
B_r(x) \cap \{ u = 0  \} = \{ (y',h(y')) : y' \in U \} .
\end{equation}
Now let $\xi \in C_0^\infty(B_r(x))$ be such that $0 \leq \xi \leq 1$ and  $\xi \equiv 1 $ on $B_{\frac{r}{2}}(x)$. Then we can infer just like in the derivation of \eqref{eq:regula} that for each $z \in B_{\frac{r}{4}}(x)$ we have 
\begin{equation}\label{eq:559}
u(z)\xi(z) =  -\frac{1}{2}\int F(z,y) \xi(y)\; \mathrm{d}\mu(y) + h(z) ,
\end{equation}
where $h \in C^\infty(\overline{B_\frac{r}{4}(x)})$ and $F$ is given in Lemma \ref{lem:green}. Recall that by \eqref{eq:542} $\mu$ can be characterized further. 
Given this \eqref{eq:559} simplifies to
\begin{equation*}
u(z) = h(z)  - \frac{1}{2} \int F(z,y) \xi(y) \frac{1}{|\nabla u(y)  |} \; \mathrm{d}\mathcal{H}^1  \llcorner_{ \{ u = 0  \} }(y) \quad \forall z \in B_{\frac{r}{4}}(x) .
\end{equation*}
Using the graph reparametrization \eqref{eq:graphr} we get 
\begin{equation}\label{eq:estii}
u(z) = h(z) - \frac{1}{2} \int_U F(z,(y', h(y')) \xi(y',h(y')) \frac{1}{|\nabla u|(y',h(y'))  } \sqrt{ 1 +| \nabla h (y')|^2} \dy'.
\end{equation}
By choice of $r$ we have $| (y', h(y')) - x| < r < \frac{1}{2}$ for each $y' \in U$ and $|z-x|< \frac{r}{4}< \frac{1}{8}$. Hence $| z- (y', h(y')) | < \frac{5}{8} <1$ which implies that the expression in the integral is negative, see the properties of $F$ in Lemma \ref{lem:green}.
Taking the derivative using similar techniques as in the proof of Corollary \ref{cor:regreg} we get 
\begin{align*}
  \Delta u(z)  & = \Delta h (z) - \frac{1}{4\pi} \int_{U} \bigg( \left( 1 + \frac{1}{2} \log (| z_1 - y'|^2 + |z_2 - h(y')|^2 ) \right)  \\ &   \quad \quad \quad \quad   \quad \quad \quad \quad \quad \quad \quad \quad \quad \quad\frac{\xi(y', h(y'))}{|\nabla u | (y', h(y'))} \sqrt{ 1 +| \nabla h (y')|^2}\bigg) \dy'
\end{align*}
Later we will use the dominated convergence theorem to show continuity. To do so,
we substitute $y'' = z_1 - y'$ to get for each $z \in B_{\frac{r}{4}}(x) $ 
\begin{align}
\Delta u(z) = \Delta h (z) -  & \frac{1}{4\pi} \int \bigg( \left( 1 + \frac{1}{2} \log (|y''|^2 + |z_2 - h(z_1 - y'')|^2 ) \right) \nonumber \\ & \quad \quad  \chi_{z_1 - U}(y'')  \frac{\xi(z_1- y'', h(z_1 - y''))}{|\nabla u | (z_1 - y'', h(z_1 - y''))} \sqrt{ 1 +| \nabla h (z_1 - y'')|^2} \bigg)\; \mathrm{d}y''\label{eq:567}
\\ =: \Delta h(z) - & \int g(z,y'') \dy''.\nonumber 
\end{align}
Now suppose that $\left(z^{(n)} \right)_{n = 1}^\infty = \left(  (z_1^{(n)}, z_2^{(n)}) \right)_{n = 1}^\infty \subset B_{\frac{r}{4}}(x)$ is a sequence such that $z^{(n)} \rightarrow x$. By monotonicity of the logarithm and an argument similar to the discussion after \eqref{eq:estii}  we have 
\begin{equation*}
1 + \frac{1}{2} \log |y''|^2 \leq  \left( 1 + \frac{1}{2} \log (|y''|^2 + |z_2 - h(z_1 - y'')|^2 ) \right) \leq 1 .
\end{equation*}
Hence, 
\begin{equation}\label{eq:logbound}
\left\vert 1 + \frac{1}{2} \log (|y''|^2 + |z_2 - h(z_1 - y'')|^2 ) \right\vert \leq  1 + | \log |y''| | .
\end{equation}
Moreover, 
\begin{equation*}
y'' \mapsto \left\vert  \chi_{z_1 - U}(y'')  \frac{\xi(z_1- y'', h(z_1 - y''))}{|\nabla u | (z_1 - y'', h(z_1 - y''))} \sqrt{ 1 +| \nabla h (z_1 - y'')|^2} \right\vert 
\end{equation*}
can be bounded independently of $z \in B_\frac{r}{4}(x)$ by  $C \chi_{B_A(0)}(y'')$ for some appropriate $C,A > 0 $. 
Given this and \eqref{eq:logbound} the integrand in \eqref{eq:567} can be bounded by $C(1 + | \log |y''| |) \chi_{B_A(0)}(y'')$, which is an integrable dominating function. 
By the dominated convergence theorem we can interchange pointwise a.e. limits and integration. Since $g(z^{(n)}, y'') \rightarrow g(z,y'')$ for Lebesgue almost every $y'' \in \mathbb{R}$ we obtain that $\Delta u (z^{(n)}) \rightarrow \Delta u (x)$ which shows the desired continuity.
\end{proof}
\section{Vanishing radial limits}\label{sec:vanrad}
 We assume for this appendix section that $\Omega = B_1(0)$. We will first show existence of the radial limits and then improve upon this result by showing that they vanish.
 \begin{lemma}[Existence of Radial Limits]\label{lem:radlim} 
Let $\Omega = B_1(0)$ and $u \in \A(u_0)$ be a minimizer. Let $(r, \theta) \in [0,1) \times [0,2\pi)$ be the polar coordinate representation of $\Omega$. Then for almost every $\theta \in [0,2\pi]$ there exists
\begin{equation*}
\lim_{r \rightarrow 1} \Delta u (r, \theta ) := B(\theta) .
\end{equation*} 
Moreover $B \in L^1(0,2\pi)$.
\end{lemma}
\begin{proof}
We apply \cite[Main Theorem]{Littlewood}, which says that any subharmonic function $v$ in $B_1(0)$ such that 
\begin{equation}\label{eq:avin}
[0, 1) \ni r \mapsto \int_0^{2\pi} |v(r, \theta) | \; \mathrm{d}\theta \in \mathbb{R} \quad \textrm{is bounded}
\end{equation}
 has radial limits in $\mathbb{R}$ for almost every $\theta \in [0,2\pi)$  as $r \rightarrow 1$, i.e. there is a measurable map $L = L(\theta) : [0, 2\pi ) \rightarrow \mathbb{R}$ such that 
\begin{equation*}
\lim_{r \rightarrow 1} v(r, \theta ) = L(\theta) \quad\textrm{for almost every} \; \theta \in [0,2\pi).
\end{equation*}
To show existence of radial limits of $\Delta u$, we check \eqref{eq:avin} for $v = -\Delta u$, which is subharmonic by \cite[Theorem 4.3]{Serrin}, as it is continuous by Theorem \ref{thm:1.1} and weakly subharmonic by Lemma \ref{lem:EL}.   Note that $|-\Delta u (0)|$ is a finite number as $0$ lies in the interior if $B_1(0)$. By superharmonicity of $\Delta u$ and $\Delta u \geq 0 $ by Corollary \ref{lem:corsubham} we get
\begin{equation*}
|-\Delta u(0)| = \Delta u(0) \geq \frac{1}{2\pi} \int_0^{2\pi} \Delta u (r , \theta) \; \mathrm{d}\theta = \frac{1}{2\pi} \int_0^{2\pi} |\Delta u (r, \theta) | \; \mathrm{d}\theta   \quad  \forall r \in (0,1), 
\end{equation*} 
which implies that $v = - \Delta u$ fulfills \eqref{eq:avin} and hence the existence of radial limits of $\Delta u$. 
 Define for almost every $\theta \in [0, 2\pi)$. 
\begin{equation*}
B(\theta) := \lim_{r \rightarrow 1}  \Delta u(r, \theta) .
\end{equation*}
Notice that $B (\theta ) \geq  0$. By Fatou's Lemma we have 
\begin{align}\label{eq:86}
\frac{1}{2\pi} \int_0^{2\pi}  B(\theta)  \; \mathrm{d}\theta & \leq \liminf_{ r \rightarrow 1 } \frac{1}{2\pi} \int_0^{2\pi} \Delta u(r, \theta)  \; \mathrm{d}\theta  = \liminf_{ r \rightarrow 1 } \frac{1}{2\pi r} \int_0^{2\pi} \Delta u(r, \theta) \; \mathrm{d}\theta \nonumber \\ & = \liminf_{ r \rightarrow 1 }  \fint_{\partial B_r(0) } \Delta u(x) \; \mathrm{d}S_r(x) ,
\end{align}
where $\mathrm{d}S_r$ denotes the surface measure on $\partial B_r(0)$. Since $\Delta u$ is superharmonic, the integral average is bounded from above by $\Delta u (0)$ which is a finite number. Therefore $B \in L^1(0,2\pi)$ and the claim is shown. 
\end{proof}
\begin{remark}
One can infer from \eqref{eq:86} that 
\begin{equation}\label{eq:88}
\frac{1}{2\pi} \int_0^{2\pi}  B(\theta) \;  \mathrm{d}\theta \leq \liminf_{ r \rightarrow 1 }  \fint_{\partial B_r(0) } \Delta u(x) \; \mathrm{d}S_r(x)  = \inf_{r \in (0,1) } \fint_{\partial B_r(0) } \Delta u(x) \; \mathrm{d}S_r(x), 
\end{equation}
where the last equality holds because of superharmonicity of $\Delta u$ (cf. proof of Lemma \ref{lem:radlim}) and \cite[Proposition 4.4.15]{Berenstein}. 
\end{remark}

\begin{lemma}[Strong Navier Boundary Conditions] \label{lem:strongnav}
Let $u \in \A(u_0)$ be a minimizer in $\Omega = B_1(0)$. Let $B(\theta)$ be defined as in Lemma \ref{lem:radlim}. Then $B(\theta) = 0 $ for almost every $\theta \in [0,2\pi)$. In particular the radial limits of $\Delta u$ exist $\mathcal{H}^1$ almost everywhere and equal zero. 
\end{lemma}
\begin{proof}
We show that 
\begin{equation}\label{eq:dingenss}
\frac{1}{2\pi }\int_0^{2\pi} B(\theta) \; \mathrm{d}\theta = 0  .
\end{equation}
The claim follows from this since $B$ is nonnegative (recall Lemma \ref{lem:radlim} and Corollary \ref{lem:corsubham}). Now suppose the opposite, i.e.  
\begin{equation*}
\frac{1}{2\pi }\int_0^{2\pi} B(\theta) \; \mathrm{d}\theta > 0  .
\end{equation*}
 Let $f \in C([0,1])$ and consider the Poisson problem 
\begin{equation*}
\begin{cases}
\Delta \phi (x) = f(|x|) &  x \in  B_1(0) \\
\phi(x) = 0 & x \in \partial B_1(0) 
\end{cases}.
\end{equation*}
By elliptic regularity, see \cite[Theorem 8.12]{Gilbarg}, the problem has a solution $\phi \in W^{2,2}(\Omega) \cap W_0^{1,2}(\Omega)$. Now observe using \eqref{eq:89}, \eqref{eq:88}, spherical integration and Fubini's Theorem: 
\begin{align*}
||\phi||_{L^\infty(\Omega) } \frac{|\Omega|}{\inf_{\partial B_1(0)} u_0} & \geq \int_{\{ u = 0 \}} \frac{-\phi}{|\nabla u|} \; \mathrm{d}\mathcal{H}^1 = \int_{B_1(0)} \Delta u \Delta \phi \dx = \int_{B_1(0)} \Delta u(x)  f(|x|) \dx 
\\ & = \int_0^1 f(r) \int_{\partial B_r(0)} \Delta u \; \mathrm{d}S_r \dr = 2\pi \int_0^1 r f(r) \fint_{\partial B_r(0)} \Delta u \; \mathrm{d}S_r \dr
\\ & \geq 2\pi  \int_0^1 r f(r) \frac{1}{2\pi} \int_{0}^{2\pi} B(\theta)\; \mathrm{d}\theta  \dr \\ & = \left( \int_0^1 2\pi r f(r) \dr  \right)  \left(  \frac{1}{2\pi}\int_0^{2\pi} B(\theta) \; \mathrm{d}\theta \right) 
\\ & = \left( \int_{B_1(0)} f(|x|) \dx  \right)  \left(  \frac{1}{2\pi}\int_0^{2\pi} B(\theta) \; \mathrm{d}\theta \right)
\\ & =  \left( \int_{B_1(0)} \Delta \phi(x) \dx  \right)  \left(  \frac{1}{2\pi}\int_0^{2\pi} B(\theta) \; \mathrm{d}\theta \right). 
\end{align*}
This implies that for each $\phi \in C^2(\overline{B_1(0)}) \cap W_0^{1,2}(B_1(0))$ that is radial one has 
\begin{equation}\label{eq:824}
\left( \int_{B_1(0)} \Delta \phi(x) \dx  \right) \leq  \left(  \frac{1}{2\pi}\int_0^{2\pi} B(\theta) \; \mathrm{d}\theta \right)^{-1} \frac{|\Omega|}{\inf_{\partial B_1(0)} u_0} ||\phi||_{L^\infty(B_1(0))} .
\end{equation}
Now consider for arbitrary $p \in [2,\infty)$, $\phi_p(x) := |x|^p -1 $. Observe that $\phi_p$ is radial, lies in $C^2 (\overline{B_1(0)}) \cap W_0^{1,2}(B_1(0))$, that $||\phi_p||_{L^\infty(B_1(0))} = 1$ and  $\Delta \phi_p = p^2 |x|^{p-2}$. Therefore 
\begin{equation*}
\left( \int_{B_1(0)} \Delta \phi_p(x) \dx  \right)  = \int_0^1 2\pi p^2 r r^{p-2} \dr  = 2\pi p .
\end{equation*}
Plugging this into  \eqref{eq:824} we obtain
\begin{equation*}
2\pi p  \leq  \left(  \frac{1}{2\pi}\int_0^{2\pi} B(\theta) d\theta \right)^{-1} \frac{|\Omega|}{\inf_{\partial B_1(0)} u_0} .
\end{equation*}
We can obtain a contradiction choosing $ p : = \frac{1}{\pi}\left(  \frac{1}{2\pi}\int_0^{2\pi} B(\theta) d\theta \right)^{-1} \frac{|\Omega|}{\inf_{\partial B_1(0)} u_0} $, which is an admissible choice. Hence we have shown \eqref{eq:dingenss}  by contradiction. 
\end{proof}

\begin{remark}
Note also that by Corollary \ref{cor:guterrand} and Lemma \ref{lem:biham}, $\Delta u$ is harmonic in the Annulus $B_1(0) \setminus \overline{B}_{r_0}(0)$ for some $r_0$ sufficiently close to $1$. By \cite[Theorem 2]{Dahlberg}, $\Delta u$ has also nontangential limits as $x \rightarrow e^{i \theta} \in \partial B_1(0)$ for almost every $\theta \in [0, 2\pi)$. Of course the nontangential limit has to coincide with the radial limit, which is zero as we just proved. Hence  we obtain `$\Delta u = 0$ $\mathcal{H}^1$-almost everywhere on $\partial B_1(0)$' in the sense of nontangential limits, see \cite{Dahlberg} for more details on these.   
\end{remark}

%
%


\end{document}